\documentclass[reqno,centertags]{amsart}
\usepackage{amsmath,amsfonts,amsthm,amssymb}
\usepackage[a4paper]{geometry}
\usepackage{hyperref}
\newtheorem*{res}{Result}
\newtheorem{lem}{Lemma}[section]
\newtheorem{thm}[lem]{Theorem}
\newtheorem{prop}[lem]{Proposition}
\newtheorem{cor}[lem]{Corollary}
\theoremstyle{remark}
\newtheorem{expl}{Example}[section]
\numberwithin{equation}{section}
\theoremstyle{definition}
\newtheorem{df}{Definition}
\def\R{\mathbb R}
\def\C{\mathbb C}
\def\D{\mathrm D}
\def\d{\mathrm d}
\def\S{\mathbb S}
\def\i{\mathrm i}
\let\Im\relax
\DeclareMathOperator{\Im}{Im}
\let\Re\relax
\DeclareMathOperator{\Re}{Re}
\DeclareMathOperator{\spec}{spec}
\DeclareMathOperator{\diag}{diag}
\DeclareMathOperator{\bdiag}{b-diag}
\DeclareMathOperator{\trace}{tr}
\DeclareMathOperator{\supp}{supp}

\begin{document}
\title[Anisotropic thermo-elasticity in 2D]{Anisotropic thermo-elasticity in 2D \\  Part I: A unified treatment }
\author{Michael Reissig}
\address{Michael Reissig, Institut f\"ur Angewandte Analysis, Fakult\"at f\"ur Mathematik und Informatik,
TU Bergakademie Freiberg, Pr\"uferstra\ss{}e 9, 09596 Freiberg, Germany}
\author{Jens Wirth}
\address{Jens Wirth, Institut f\"ur Angewandte Analysis, Fakult\"at f\"ur Mathematik und Informatik,
TU Bergakademie Freiberg, Pr\"uferstra\ss{}e 9, 09596 Freiberg, Germany}
\address{{\sl current address:} Department of Mathematics, Imperial College, London SW7 2AZ, UK}

\date{}
%\maketitle

\begin{abstract}
In this note we develop tools and techniques for the treatment of anisotropic
thermo-elasticity in two space dimensions. We use a diagonalisation technique
to obtain properties of the characteristic roots of the full symbol of the 
system in order to prove $L^p$--$L^q$ decay rates for its solutions.

\end{abstract}

\maketitle

\section{The problem under consideration}
\label{sec:1}

Systems of thermo-elasticity are hyperbolic-parabolic or hyperbolic-hyperbolic coupled systems
(type-1, type-2 or type-3 models) describing the elastic and thermal behaviour of elastic, heat-conducting media. The classical type-1 model of thermo-elasticity is based on Fourier's law, which means, that the heat flux is proportional to the gradient of the temperature. The present paper is devoted to the study of type-1 systems for homogeneous but anisotropic media in $\R^2$. There are different results in the literature for certain anisotropic media (cubic in \cite{Borkenstein}; rhombic in \cite{Doll}). Our goal is to present an approach, which allows to consider (an)isotropic models in $\R^2$ from a unified point of view.

We consider the type-1 system of thermo-elasticity
\begin{subequations}\label{eq:System1}
\begin{align}
   U_{tt}+A(\D)U+\gamma\nabla \theta&=0,\\
   \theta_t-\kappa \Delta \theta +\gamma\nabla^TU_t&=0.
\end{align}
\end{subequations}
Here $A(\D)$ denotes the elastic operator, which is assumed to be a homogeneous second order $2\times 2$
matrix of (pseudo) differential operators and models the elastic properties 
of the underlying medium. Furthermore, $\kappa$ describes the conduction of heat
and $\gamma$ the thermo-elastic coupling of the system. We assume $\kappa>0$ and $\gamma\neq0$. 
We solve the Cauchy problem for system~\eqref{eq:System1} with initial data
\begin{equation}
  U(0,\cdot)=U_1,\qquad U_t(0,\cdot)=U_2,\qquad \theta(0,\cdot)=\theta_0,
\end{equation}
for simplicity we assume $U_1,U_2\in\mathcal S(\R^2,\R^2)$ and $\theta_0\in\mathcal S(\R^2)$.
We denote by $A(\xi)$ the symbol of the elastic operator and we set $\eta=\xi/|\xi|\in\S^1$.
Then some basic examples for our approach are given as follows. The material constants are always
specified in such a way that the matrix $A(\eta)$ becomes positive.
\begin{expl}
  {\sl Cubic} media in 2D are modelled by   
  \begin{equation}
    A(\eta) =
    \begin{pmatrix}
      (\tau-\mu)\eta_1^2+\mu & (\lambda+\mu)\eta_1\eta_2 \\ (\lambda+\mu)\eta_1\eta_2   & (\tau-\mu)\eta_2^2+\mu 
    \end{pmatrix}
  \end{equation}
  with constants $\tau,\mu>0$, $-2\mu-\tau<\lambda<\tau$. This case was treated e.g. in \cite{Borkenstein}. For the corresponding elastic system see \cite{Stoth}.
\end{expl}
\begin{expl}
  {\sl Rhombic} media in 2D are modelled by
   \begin{equation}
    A(\eta) =
    \begin{pmatrix}
      (\tau_1-\mu)\eta_1^2+\mu & (\lambda+\mu)\eta_1\eta_2 \\ (\lambda+\mu)\eta_1\eta_2   & (\tau_2-\mu)\eta_2^2+\mu 
    \end{pmatrix}
  \end{equation}
  with constants $\tau_1,\tau_2,\mu>0$ and $-2\mu-\sqrt{\tau_1\tau_2}<\lambda<\sqrt{\tau_1\tau_2}$.   For this case we refer also to \cite{Doll}.
\end{expl}
\begin{expl}
  Although it is not the main point of this note, we can consider
  {\sl isotropic} media, where
  \begin{align}
    A(\eta)&=\mu I+(\lambda+\mu)\eta\otimes\eta \notag\\
    &=\begin{pmatrix}
      (\lambda+\mu)\eta_1^2+\mu & (\lambda+\mu)\eta_1\eta_2 \\ (\lambda+\mu)\eta_1\eta_2   & (\lambda+\mu)\eta_2^2+\mu 
    \end{pmatrix}
  \end{align}
  with Lam\'e constants $\mu>0$ and $\lambda+\mu>0$.
\end{expl}

We will present a unified treatment of these cases of (in general) 
anisotropic thermo-elasticity. For this we assume that the homogeneous symbol
$A=A(\xi)=|\xi|^2A(\eta)$, $\eta=\xi/|\xi|$, is given as a function 
\begin{equation}
  A\;:\;\S^1\to\C^{2\times 2}
\end{equation}
subject to the conditions
\begin{description}
\item[(A1)] $A$ is real-analytic in $\eta\in\S^1$,
\item[(A2)] $A(\eta)$ is self-adjoint and positive for all $\eta\in\S^1$.
\end{description}
For some results we require that
\begin{description}
\item[(A3)] $A(\eta)$ has two distinct eigenvalues, $\#\spec A(\eta)=2$.
\end{description}
Under assumption (A3) the direction $\eta\in\S^1$ is called (elastically) {\em non-degenerate}. 
In this case we know that the elasticity equation $U_{tt}+A(\D)U=0$ is strictly
hyperbolic and can be diagonalised smoothly using a corresponding system
of normalised eigenvectors $r_j(\eta)$ to the eigenvalues $\varkappa_j(\eta)$ of $A(\eta)$.

If (A3) is violated we will call the corresponding directions $\eta\in\S^1$ {\em degenerate}. For these
directions we can use the one-dimensionality of $\S^1$ in 
connection with the analytic perturbation theory of self-adjoint matrices
(cf.\!\! \cite{Kato}). So we can always find locally smooth eigenvalues $\varkappa_j(\eta)$
and corresponding locally smooth normalised eigenvectors $r_j(\eta)$ of $A(\eta)$.
For the following we assume for simplicity that these functions extend to 
global smooth functions on $\S^1$.

This classification of directions is not sufficient for a precise study of $L^p$--$L^q$ decay estimates
for solutions to the thermo-elastic system. It turns out that different microlocal directions $\eta=\xi/|\xi|\in\S^1$
from the phase space have different influence on decay estimates. But how to distinguish these directions and how to understand their influence? In general, this can be done by a refined diagonalisation procedure applied to a corresponding
first order system (first order with respect to time). Applying a partial Fourier transform and chosing a suitable energy (of minimal dimension) this system reads as $\D_t V=B(\xi)V$. The properties of the matrix $B(\xi)$ are essential for our understanding:
\begin{itemize}
\item the notions of {\em hyperbolic} and {\em parabolic directions} depend on the behaviour of the eigenvalues of $B(\xi)$ (see \eqref{eq2.3}, \eqref{eq2.5}, Definition~\ref{df:hyp});
\item the matrix $B(\xi)$ contains spectral data of $A(\xi)$ together with certain {\em coupling functions} (see \eqref{eq2.6}) between different components of the energy. The behaviour of these coupling functions close to hyperbolic directions has an essential influence on decay rates (see
Theorems~\ref{thm3.1}, \ref{thm3.5} and \ref{thm3.6}).
\end{itemize}
It turns out that we have to exclude some exceptional values of the coupling constant $\gamma$ by assuming that (see Definition~\ref{df:hyp} for the notion of hyperbolic directions)
\begin{description}
\item[(A4)]  $\gamma^2\ne2 \varkappa_{j_0}(\bar\eta)-\trace A(\bar\eta)$ for all hyperbolic directions
$\bar\eta$ with respect to  $\varkappa_{j_0}$.
\end{description}
Basically this implies the non-degeneracy of the $1$-homogeneous part of $B(\xi)$. In the following we will call a hyperbolic direction violating (A4) a {\em $\gamma$-degenerate} direction. Assumption (A4) is used for the treatment of small hyperbolic frequencies and plays there a similar r\^ole like (A3) for large frequencies.

In Section~\ref{sec:2} we will give the transformation of the thermo-elastic system
\eqref{eq:System1} to a system of first order and the diagonalisation procedure in detail.
The proposed procedure generalises those from \cite{Wang}, \cite{Wang2}, \cite{ReiWang}, \cite{Yagdjian}. 
The obtained results are used to represent solutions of the original system as Fourier integrals
with complex phases. Based on these representations we give micro-localised decay estimates for solutions in Section~\ref{sec:3}. They and their method of proof depend on
\begin{itemize}
\item the classification of directions (to be hyperbolic or parabolic);
\item the order of contact between Fresnel curves (coming from the elastic part) and their tangents
for hyperbolic directions;
\item the vanishing order of the coupling functions in hyperbolic directions. 
\end{itemize}

Let us formulate some of the results. The first one follows from Theorem~\ref{thm3.1} and Corollary~\ref{cor3.1}.
\begin{res} 
  Under assumptions (A1) to (A4) and if the coupling functions vanish to first order
  in hyperbolic directions the solutions $U(t,x)$ and $\theta(t,x)$ to \eqref{eq:System1} 
  satisfy the  $L^p$--$L^q$ estimate
  \begin{multline}
    \|\D_t U(t,\cdot)\|_q+\|\sqrt{A(\D)} U(t,\cdot)\|_q + \|\theta(t,\cdot)\|_q \\
    \lesssim     (1+t)^{-\frac12(\frac1p-\frac1q)} \left(\|U_1\|_{p,r+1} + \|U_2\|_{p,r} + \|\theta_0\|_{p,r} \right)
  \end{multline}
  for dual indices $p\in(1,2]$, $pq=p+q$, and Sobolev regularity $r>2(1/p-1/q)$.
\end{res}
If the coupling functions vanish to higher order we have to relate their vanishing order $\ell$ to the order of contact $\bar\gamma$ between the Fresnel curve and its tangent in the corresponding direction and for the corresponding sheet. 
In Theorems~\ref{thm3.5} and \ref{thm3.6} we show that in this case the $1/2$ in the exponent is changed to $1/\min(2\ell,\bar\gamma)$.

\medskip
Our main motivation to write this paper is to provide a unified way to treat anisotropic models of thermo-elasiticity. New analytical tools presented in these notes generalise to higher dimensions and allow to treat especially models in 3D (outside degenerate directions).  The two-dimensional results from \cite{Borkenstein} for cubic media and \cite{Doll} for rhombic media are contained / extended; general anisotropic media can be treated. This is discussed in some detail in the second part \cite{WirMedia}
of this note.
 
\medskip
\noindent
{\bf Acknowledgements.} The first author thanks Prof. Wang Ya-Guang (Shanghai Jiao Tong University) for the
discussion about some basic ideas of the approach presented in this paper during his stay at the TU Bergakademie
Freiberg in August 2004. The stay was supported by the German-Chinese research project 446 CHV 113/170/0-2.

\section{General treatment of the thermo-elastic system}
\label{sec:2}

We use a partial Fourier transform with respect to the spatial variables
to reduce the Cauchy problem for \eqref{eq:System1} to the system of ordinary
differential equations
\begin{subequations}\label{eq:System2}
  \begin{align}
    &\hat U_{tt}+|\xi|^2A(\eta)\hat U+\i\gamma\xi\hat\theta=0,\\
    &\hat\theta_t+\kappa|\xi|^2 \hat\theta+\i\gamma\xi\cdot\hat U_t =0, \\
    &\hat U(0,\cdot)=\hat U_1,\quad \hat U_t(0,\cdot)=\hat U_2,\quad \hat\theta(0,\cdot)=\hat\theta_0
  \end{align}
\end{subequations}
parameterised by the frequency variable $\xi$.

We denote by $\varkappa_1,\varkappa_2\in C^\infty(\S^1)$ the eigenvalues of $A(\eta)$ and by $r_1,r_2\in C^\infty(\S^1,\S^1)$
corresponding normalised eigenvectors. Both depend in a real-analytic way on 
$\eta\in\S^1$.
In a first step we reduce \eqref{eq:System2} to a first order system. For this
we use the diagonaliser of the elastic operator, i.e. the matrix 
$M(\eta)=(r_1(\eta)|r_2(\eta))$ build up from the normalised eigenvectors, and define
\begin{equation}
  U^{(0)}(t,\xi)=M^T(\eta)\hat U(t,\xi).
\end{equation}
Then we define by the aid of
\begin{equation}\label{eq2.3}
  \mathcal D^{1/2}(\eta)=\diag(\omega_1(\eta),\omega_2(\eta)),\qquad\omega_j(\eta)=\sqrt{\varkappa_j(\eta)}\in C^\infty(\S^1),
\end{equation}
the vector-valued function
\begin{equation}\label{eq:hatVdef}
  V^{(0)}(t,\xi)=
  \begin{pmatrix}
    (\D_t+\mathcal D^{1/2}(\xi)) U^{(0)}(t,\xi)\\
    (\D_t-\mathcal D^{1/2}(\xi)) U^{(0)}(t,\xi)\\
    \hat\theta(t,\xi)
  \end{pmatrix},
\end{equation}
where as usual $\D_t=-\i\partial_t$.
It satisfies a first order system with apparently simple structure. A short
calculation yields $\D_t V^{(0)}(t,\xi)=B(\xi) V^{(0)}(t,\xi)$ with
\begin{equation}\label{eq2.5}
  B(\xi)=
  \begin{pmatrix}
    \omega_1(\xi) & & & & \i\gamma a_1(\xi) \\
    & \omega_2(\xi) & & & \i\gamma a_2(\xi) \\ 
    & & -\omega_1(\xi) & & \i\gamma a_1(\xi) \\ 
    & & & -\omega_2(\xi) & \i\gamma a_2(\xi) \\
    -\frac{\i\gamma}2 a_1(\xi) &  -\frac{\i\gamma}2 a_2(\xi) & 
    -\frac{\i\gamma}2 a_1(\xi) &  -\frac{\i\gamma}2 a_2(\xi) & \i\kappa|\xi|^2
  \end{pmatrix},
\end{equation}
where we used the {\em coupling functions}
\begin{equation}\label{eq2.6}
  a_j(\xi)=r_j(\eta)\cdot\xi.
\end{equation}
For later use we introduce the notation $B_1(\xi)$ and $B_2(\xi)$ for the
homogeneous components of $B(\xi)$ of order $1$ and $2$, respectively.
The coupling functions $a_j(\eta)$ can be understood as the co-ordinates of $\eta$ with respect 
to the orthonormal eigenvector basis $\{r_1(\eta),r_2(\eta)\}$. Therefore, it holds $a_1^2(\eta)+a_2^2(\eta)=1$.
Furthermore, they are well-defined and real-analytic functions on $\S^1$.

In the following proposition we collect some information on the characteristic
polynomial of the matrix $B(\xi)$.
\begin{prop}\label{prop:2.0}
  \begin{enumerate}
  \item $\trace B(\xi)=\i\kappa|\xi|^2$ and $\det B(\xi)=\i\kappa|\xi|^6\det A(\eta)$.
  \item The characteristic polynomial of $B(\xi)$ is given by
    \begin{align}\label{eq:2.7}
      \det(\nu I-B(\xi)) =& (\nu-\i\kappa|\xi|^2)(\nu^2-\varkappa_1(\xi))(\nu^2-\varkappa_2(\xi))\notag\\
      &-\nu\gamma^2 |\xi|^2a_1^2(\eta)(\nu^2-\varkappa_2(\xi))-\nu\gamma^2|\xi|^2a_2^2(\eta)(\nu^2-\varkappa_1(\xi)).
    \end{align}
  \item An eigenvalue $\nu\in\spec B(\xi)$, $\xi\neq0$, is real if and only if 
    $\nu^2=\varkappa_{j_0}(\xi)$ for an index $j_0=1,2$. If the direction is
    non-degenerate this is equivalent to $a_{j_0}(\eta)=0$.
%  \item If $\eta\in\S^1$ is degenerate, the characteristic roots of $B(\xi)$ are determined by $\nu^2=\varkappa(\xi)$ and
 % $ (\nu-\i\kappa|\xi|^2)(\nu^2-\varkappa(\xi))=\nu\gamma^2|\xi|^2$.
 \item If $a_j(\eta)\ne 0$, $j=1,2$ the eigenvalues $\nu\in\spec B(\xi)$ satisfy
 \begin{equation}
   1=\frac{\i\kappa|\xi|^2}\nu + \gamma^2\frac{a_1^2(\xi)}{\nu^2-\varkappa_1(\xi)}+ 
   \gamma^2\frac{a_2^2(\xi)}{\nu^2-\varkappa_2(\xi)}.
 \end{equation}
  \end{enumerate}
\end{prop}

It turns out that the property of $B(\xi)$ to have real eigenvalues
depends only on the direction $\eta=\xi/|\xi|\in\S^1$. We will introduce a notation. 
\begin{df}\label{df:hyp}
We call a direction $\eta\in\S^1$ {\em hyperbolic} if $B(\xi)$ has a real eigenvalue
and {\em parabolic} if all eigenvalues of $B(\xi)$ have non-zero imaginary part.
\end{df}
In hyperbolic directions we always have a pair of real eigenvalues. 
If $\eta\in\S^1$ is hyperbolic with $\pm\omega_{j_0}(\xi)\in\spec B(\xi)$ for $\xi=|\xi|\eta$, we call
$\eta$ hyperbolic {\em with respect to the index $j_0$} (or with respect to the eigenvalue
$\varkappa_{j_0}(\eta)$ of $A(\eta)$) and $\nu_\pm(\xi)=\pm\omega_{j_0}(\xi)$ the corresponding pair of {\em hyperbolic eigenvalues} of $B(\xi)$. 

A non-degenerate direction is parabolic if and only if $a_j(\eta)\neq0$, $j=1,2$, 
while for  non-degenerate hyperbolic directions one of the coupling functions
$a_{j_0}(\eta)$ vanishes. Degenerate directions are always hyperbolic (in 2D), see \eqref{eq:2.7}.

\begin{expl}
  If the medium is isotropic, $A(\eta)=\mu I+(\lambda+\mu)\eta\otimes\eta$, the eigenvalues of $A$
  are $\mu$ and $\lambda+\mu$ with corresponding eigenvectors $\eta$ and $\eta^\perp$. Thus all
  directions are hyperbolic (with respect to the second eigenvalue). 
  In this case the matrix $B(\xi)$ decomposes
  into a diagonal hyperbolic $2\times2$-block and a parabolic $3\times3$-block. This
  decomposition coincides with the Helmholtz decomposition as used in the 
  standard treatment of isotropic thermo-elasticity.
\end{expl}
\begin{expl}
  For cubic media (where we assume in addition $\mu\neq\tau$ and $\mu+\lambda\neq0$)
  there exist eight hyperbolic directions determined by $\eta_1\eta_2=0$ or
  $\eta_1^2=\eta_2^2$. The functions $a_j(\eta)$ have simple zeros at these directions.
\end{expl}
\begin{expl}
  Weakly coupled cubic media with $\lambda+\mu=0$, $\mu\ne\tau$, have the degenerate 
  directions $\eta_1^2=\eta_2^2$, media with $\mu=\tau$, $\lambda+\mu\ne0$, for $\eta_1\eta_2=0$. 
  In both cases the coupling functions $a_j(\eta)$ do not vanish in these directions.
  
  If $\mu=\tau=-\lambda$, the elastic system decouples directly into two wave equations with
  propagation speed $\sqrt\mu$. In this case
  all directions are degenerate.
\end{expl}
\begin{expl}\label{expl:rhombic}
  For rhombic media we have to distinguish between three cases. \\
  {\sl Case 1.} If the material constants satisfy $(\lambda+2\mu-\tau_1)(\lambda+2\mu-\tau_2)>0$, 
  we are close to the cubic case and
  there exist eight hyperbolic directions given by $\eta_1\eta_2=0$ and
  $$ \eta_1^2(\lambda+2\mu-\tau_1)=\eta_2^2(\lambda+2\mu-\tau_2). $$
  {\sl Case 2.} If we assume on the contrary that 
  $(\lambda+2\mu-\tau_1)(\lambda+2\mu-\tau_2)<0$, only the four hyperbolic directions
  $\eta_1\eta_2=0$ exist. In the Cases 1 and 2 in each hyperbolic direction one of the 
  coupling functions $a_j(\eta)$ vanishes to first order.\\
  {\sl Case 3.} In the borderline case $\tau_1=\lambda+2\mu$ or $\tau_2=\lambda+2\mu$, but
  $\tau_1\ne\tau_2$, three hyperbolic directions collapse to one. We have the four 
  hyperbolic directions $\eta_1\eta_2=0$, at two of them ($\eta_j=\pm1$ if $\tau_j=\lambda+2\mu$)
  the vanishing order of  the coupling function
  is three.\\
  Rhombic media are degenerate if {\sl a)} $\mu=\tau_1$ (or $\mu=\tau_2$) with degenerate direction
  $(0,1)^T$ (or $(1,0)^T$) or {\sl b)}  $\lambda+\mu=0$ (weakly coupled case) and 
  $(\mu-\tau_1)(\mu-\tau_2)>0$ with degenerate directions determined by 
  $\eta_1^2(\mu-\tau_1)=\eta_2^2(\mu-\tau_2)$ or {\sl c)}  $\tau_i=\mu=-\lambda$ (exceptional case)
  where all directions are
  degenerate.
\end{expl}

\begin{prop}\label{prop:2.0.1}
  Let the direction $\bar\eta\in\S^1$ be non-degenerate and hyperbolic with respect 
  to the index $j_0$.
  Then the corresponding eigenvalues $\nu_\pm(\xi)$ satisfy
  \begin{equation}
    \lim_{\eta\to\bar\eta} \frac{a_{j_0}^2(\xi)}{\nu^2_\pm(\xi)-\varkappa_{j_0}(\xi)} = q_{\bar\eta}(|\xi|)=C_{\bar\eta}\mp\i D_{\bar\eta}|\xi|  
  \end{equation}
  for all non-tangential limits with real constants $C_{\bar\eta},D_{\bar\eta}\in\R$, $D_{\bar\eta}>0$.
  Furthermore, the imaginary part of the hyperbolic eigenvalue satisfies
  \begin{equation}\label{eq:limitHypEig}
    \lim_{\eta\to\bar\eta}\frac{\Im \nu_\pm(\xi)}{a_{j_0}^2(\eta)}=\frac{\gamma^2}{2\kappa}\frac{D_{\bar\eta}^2|\xi|^2}{C_{\bar\eta}^2+|\xi|^2D_{\bar\eta}^2}>0,
  \end{equation}
  and thus vanishes like $\Im \nu(|\xi|\eta)\sim a_{j_0}^2(\eta)$ as $\eta\to\bar\eta$ for all $\xi\neq0$.
\end{prop}
\begin{proof}
  Let for simplicity $j_0=1$. We use the characteristic polynomial of $B(\xi)$ 
  to deduce 
  \begin{align}
    \frac{\gamma^2|\xi|^2a_1^2(\eta)}{\nu^2_\pm-|\xi|^2\varkappa_1(\eta)} 
    &= 1-\frac{\i\kappa|\xi|^2}{\nu_\pm}-\frac{\gamma^2|\xi|^2a_2^2(\eta)}{\nu_\pm^2-|\xi|^2\varkappa_2(\eta)}\notag\\
    &\to 1\mp\frac{\i\kappa|\xi|}{\omega_1(\bar\eta)}-\frac{\gamma^2}{\varkappa_1(\bar\eta)-\varkappa_2(\bar\eta)}\\
    &=\underbrace{1-\frac{\gamma^2}{\varkappa_1(\bar\eta)-\varkappa_2(\bar\eta)}}_{\gamma^2C_{\bar\eta}}\mp\i\underbrace{\frac{\kappa}{\omega_1(\bar\eta)}}_{\gamma^2D_{\bar\eta}}|\xi|
    =\gamma^2q_{\bar\eta}( |\xi| ).
  \end{align}
  The existence of the limit is implied by $\nu_\pm\neq0$ and $\nu_\pm^2\neq \varkappa_2(\xi)$ as consequence of
  $\nu_\pm(\xi)\to\pm|\xi|\omega_1(\eta)$ as $\eta\to\bar\eta$ by Proposition~\ref{prop:2.0}.
  
  Obviously, $\Im q_{\bar\eta}( |\xi| )= \mp\kappa|\xi|/\gamma^2\omega_1(\bar\eta)=\mp D_{\bar\eta}|\xi|$ is non-zero for $\xi\neq0$ and considering the 
  imaginary part of the first limit expression
  \begin{align*}
    \Im q_{\bar\eta}( |\xi| )&=  \lim_{\eta\to\bar\eta} \Im\frac{a_{1}^2(\xi)}{\nu^2_\pm(\xi)-\varkappa_{1}(\xi)}\\
    &= \lim_{\eta\to\bar\eta}\frac{-2 \Re \nu_\pm(\xi)\, \Im \nu_\pm(\xi)\,a_{1}^2(\xi)}{|\nu^2_\pm(\xi)-\varkappa_{1}(\xi)|^2} \\
    &= \mp 2\omega_1( |\xi|\bar\eta) |q_{\bar\eta}( |\xi| )|^2 \lim_{\eta\to\bar\eta}\frac{\Im \nu_\pm(|\xi|\eta)}{|\xi|^2a_{1}^2(\eta)}
  \end{align*}
  proves the second statement, $\lim_{\eta\to\bar\eta}\frac{\Im \nu_\pm(\xi)}{a_{1}^2(\eta)}=\frac{D_{\bar\eta}|\xi|^2}{2\omega_1(\bar\eta)(C_{\bar\eta}^2+|\xi|^2D_{\bar\eta}^2)}.$
\end{proof}

In the case of isolated degenerate directions (others are not of interest, because then the system is decoupled) we can find a replacement for Proposition~\ref{prop:2.0}. 

\begin{prop}\label{prop:2.0a}
Let $\bar\eta\in\S^1$ be an isolated degenerate direction, $\varkappa_1(\bar\eta)=\varkappa_2(\bar\eta)$.
Then the corresponding hyperbolic eigenvalues $\nu_\pm(\xi)$  satisfy
  \begin{equation}\label{eq:2.13}
     \lim_{\eta\to\bar\eta} \frac{\omega_1(\xi)-\nu^2_\pm(\xi)}{\omega_1(\xi)-\omega_2(\xi)}
     = {a_1^2(\bar\eta)}>0,
  \end{equation}
and, therefore,
\begin{equation}
\lim_{\eta\to\bar\eta}\frac{\Im \nu_\pm(\xi)}{\omega_1(\xi)-\omega_2(\xi)}=0,\qquad
\lim_{\eta\to\bar\eta}\frac{\omega_1(\xi)-\Re \nu_\pm(\xi)}{\omega_1(\xi)-\omega_2(\xi)}=a_1^2(\bar\eta).
\end{equation}
\end{prop}

Thus, if $a_1(\bar\eta)\ne0$ then the eigenvalues $\nu_\pm(\xi)$ approach $\pm\omega_1(\xi)$ at the contact order between $\omega_1(\xi)$ and $\omega_2(\xi)$ (while they approach $\pm\omega_1(\xi)$ with a higher order if $a_1(\xi)=0$). 

\subsection{Asymptotic expansion of the eigenvalues as $|\xi|\to0$}

If $|\xi|$ is small the first order part $B_1(\xi)$ dominates $B_2(\xi)$, so the 
properties of the eigenvalues are governed by spectral properties of $B_1(\xi)$. 
\begin{prop}\label{prop:2.1}
  \begin{enumerate}
  \item $\trace B_1(\xi)=0$ and $\det B_1(\xi)=0$.
  \item  
    If the direction $\eta\in\S^1$ is parabolic the nonzero eigenvalues $\tilde\nu$ of
    $B_1(\eta)$ satisfy 
    \begin{equation}
      \gamma^{-2}=\frac{a_1^2(\eta)}{\tilde\nu^2-\varkappa_1(\eta)}+\frac{a_2^2(\eta)}{\tilde\nu^2-\varkappa_2(\eta)}
    \end{equation}
    and are thus real and related to $\varkappa_j(\eta)$ by
    \begin{equation}
      0<\varkappa_1(\eta)<\tilde\nu_1^2(\eta)<\varkappa_2(\eta)<\tilde\nu_2^2(\eta)
    \end{equation}
    (if $\varkappa_1(\eta)<\varkappa_2(\eta)$). 
\item If $\eta$ is non-degenerate and hyperbolic with respect to the index $1$ we
    have $\varkappa_1(\eta)=\tilde\nu_1^2(\eta)$, while for hyperbolic directions with respect to
    the index $2$ three cases occur depending on the size of the coupling constant $\gamma$:
   \begin{equation}
   \begin{cases}
      \gamma^2<\varkappa_2(\bar\eta)-\varkappa_1(\bar\eta):\quad  &\varkappa_2(\eta)=\tilde\nu_2^2(\eta),\\
      \gamma^2=\varkappa_2(\bar\eta)-\varkappa_1(\bar\eta):\quad  &\tilde\nu_1^2(\eta)=\varkappa_2(\eta)=\tilde\nu_2^2(\eta),\\
      \gamma^2>\varkappa_2(\bar\eta)-\varkappa_1(\bar\eta):\quad  &\tilde\nu_1^2(\eta)=\varkappa_2(\eta).
  \end{cases}
  \end{equation}
\item If the direction is 
    degenerate, $\varkappa_1(\eta)=\varkappa_2(\eta)$, we have the eigenvalues $\pm\sqrt{\varkappa_1(\eta)}$ and 
    $\pm\sqrt{\varkappa_1(\eta)+\gamma^2}$.
  \end{enumerate}
\end{prop}

The existence of five distinct eigenvalues of the homogeneous principal part 
$B_1(\eta)$ for all parabolic and most hyperbolic directions allows us to calculate the full asymptotic expansion of the 
eigenvalues $\nu(\xi)$ of $B(\xi)$ as $|\xi|\to0$.
We will give only the first terms in detail, but provide the whole 
diagonalisation procedure. Assumption (A4) guarantees the non-degeneracy of $B_1(\xi)$. 
Note that, even if (A4) is violated the matrix $B_1(\xi)$ is diagonalisable (as consequence of its block structure).

\begin{prop}\label{prop:2.2} 
  As $|\xi|\to0$ the eigenvalues of the matrix $B(\xi)$ behave as
  \begin{subequations}
    \begin{align}
      \nu_0(\xi)&=\i\kappa|\xi|^2b_0(\eta)+\mathcal O( |\xi|^3),\\
      \nu_j^\pm(\xi)&=\pm|\xi|\tilde\nu_j(\eta)+\i\kappa|\xi|^2b_j(\eta)+\mathcal O( |\xi|^3)
    \end{align}
  \end{subequations}
  for all non-$\gamma$-degenerate directions, where the functions $b_j\in C^\infty(\S^1)$ are given by
  \begin{equation}
    b_0(\eta)=\left(1+\frac{\gamma^2a_1^2(\eta)}{\varkappa_1(\eta)}+\frac{\gamma^2a_2^2(\eta)}{\varkappa_2(\eta)}\right)^{-1}>0
  \end{equation}
  and
  \begin{equation}
    b_j(\eta)=\left(1+\gamma^2a_1^2(\eta)\frac{\tilde\nu_j^2+\varkappa_1(\eta)}{(\tilde\nu_j^2-\varkappa_1(\eta))^2}
      +\gamma^2a_2^2(\eta)\frac{\tilde\nu_j^2+\varkappa_2(\eta)}{(\tilde\nu_j^2-\varkappa_2(\eta))^2}\right)^{-1}\geq0.
  \end{equation}
  Furthermore, $b_j(\eta)>0$ if $\eta$ is parabolic and $b_{j_0}(\eta)=0$
  if $\eta$ is hyperbolic with respect to the index $j_0$. 
\end{prop}
\begin{proof}
  We apply a diagonalisation scheme in order to extract the spectral
  information for $B(\xi)$. We assume that the eigenvalues are denoted
  such that $\varkappa_1(\eta)\leq\varkappa_2(\eta)$. 
  
  \medskip\noindent
  {\it Step 1.} By  Proposition~\ref{prop:2.1} we know that the homogeneous 
  first order part $B_1(\eta)$ has the distinct eigenvalues $\tilde\nu_0=0$ and 
  $\tilde\nu_j^\pm(\eta)=\pm\tilde\nu_j(\eta)$, which are ordered as
  $\varkappa_1(\eta)\leq\tilde\nu_1(\eta)\leq\varkappa_2(\eta)\leq\tilde\nu_2(\eta)$ (where equality holds only under the 
  exceptions stated in Proposition~\ref{prop:2.1}). 
  We denote corresponding normalised 
  and bi-orthogonal left and right eigenvectors of the matrix $B_1(\eta)$ by
  ${}_je^\pm(\eta)$ and $e_j^\pm(\eta)$. If we collect them in the matrices
  \begin{subequations}
    \begin{align}
      L(\eta)&=({}_0e(\eta)|{}_1e^+(\eta)|{}_1e^-(\eta)|{}_2e^+(\eta)|{}_2e^-(\eta)),\\
      R(\eta)&=(e_0(\eta)|e_1^+(\eta)|e_1^-(\eta)|e_2^+(\eta)|e_2^-(\eta)),
    \end{align}
  \end{subequations}
  we have $L^*(\eta)R(\eta)=I$ and 
  \begin{equation}\label{eq:Ddef}
    L^*(\eta)B_1(\eta)R(\eta))=\mathcal D_1(\eta)
    =\diag(0,\tilde\nu_1(\eta),-\tilde\nu_1(\eta),\tilde\nu_2(\eta),-\tilde\nu_2(\eta)).
  \end{equation}
  Further we get
  \begin{equation}
    L^*(\eta)B_2(\eta)R(\eta)=\i\kappa b_*(\eta)\otimes\overline{{}_*b(\eta)},
  \end{equation}
  where $b_*(\eta)$ and ${}_*b(\eta)$ are vectors collecting the last entries $b_0(\eta)$, $b_j^\pm(\eta)$ and ${}_0b(\eta)$, ${}_jb^\pm(\eta)$
  of the eigenvectors $e_0(\eta)$, $e_j^\pm(\eta)$ and ${}_0e(\eta)$, ${}_je^\pm(\eta)$, respectively.
  
  The matrix 
  \begin{equation}
    B^{(0)}(\xi)=L^*(\eta)B(\xi)R(\eta)= |\xi|\mathcal D_1(\eta)+|\xi|^2\i\kappa b_*(\eta)\otimes\overline{{}_*b(\eta)}
  \end{equation}
  is diagonalised modulo $\mathcal O( |\xi|^2)$ as $|\xi|\to0$ and has a main part
  with distinct entries. We denote $R^{(2)}(\xi)=|\xi|^2\i\kappa b_*\otimes\overline{{}_*b}$.
  
  \medskip\noindent
  {\it Step 2.} We construct a diagonaliser of $B^{(0)}(\xi)$ as $|\xi|\to0$ of the form
  \begin{equation}
    N_k(\xi)=I+\sum_{j=1}^k |\xi|^j N^{(j)}(\eta).
  \end{equation}
  For this we denote the $k$-homogeneous part of $R^{(k)}(\xi)$ by $\tilde R^{(k)}(\xi)$
  and its entries by $\tilde R^{(k)}_{ij}(\eta)$. Then we set for $k=1,2,\ldots$
  \begin{align}
    \mathcal D_{k+1}(\eta)&=\diag \tilde R^{(k+1)}(\eta),\\
    N^{(k)}(\eta) &=
    \begin{pmatrix}
      0 & \frac{\tilde R^{(k+1)}_{12}(\eta)}{d_1(\eta)-d_2(\eta)} & \cdots & \frac{\tilde R^{(k+1)}_{15}(\eta)}{d_1(\eta)-d_5(\eta)}\\
      \frac{\tilde R^{(k+1)}_{21}(\eta)}{d_2(\eta)-d_1(\eta)} & 0 &  \cdots &\frac{\tilde R^{(k+1)}_{25}(\eta)}{d_2(\eta)-d_5(\eta)}\\
      \vdots & \vdots & \ddots & \vdots \\
      \frac{\tilde R^{(k+1)}_{51}(\eta)}{d_5(\eta)-d_1(\eta)} & \frac{\tilde R^{(k+1)}_{52}(\eta)}{d_5(\eta)-d_2(\eta)} & \cdots & 0
    \end{pmatrix},
  \end{align}
  where $d_{j}(\eta)$ are the entries of $\mathcal D_1(\eta)$. By construction we have 
  the commutator relation 
   \begin{equation}
     [\mathcal D_1(\eta),N^{(k)}(\eta)] 
     = \mathcal D_{k+1}(\eta)-\tilde R^{(k+1)}(\eta),
  \end{equation}
  such that
  \begin{align*}
    R^{(k+2)}(\xi) &= B^{(0)}(\xi)N_k(\xi)-N_k(\xi)\sum_{j=1}^{k+1} |\xi|^j\mathcal D_j(\eta)\notag\\
    &=R^{(k+1)}(\xi)\\
    &\qquad +  |\xi|^k B^{(0)}(\xi)N^{(k)}(\eta) - |\xi|^k N^{(k)}(\eta) \sum_{j=1}^{k+1} |\xi|^j\mathcal D_j(\eta) - N_k(\xi) |\xi|^{k+1}\mathcal D_{k+1}(\eta)\\
    &=R^{(2)}(\xi)|\xi|^kN^{(k)}(\eta)- |\xi|^k N^{(k)}(\eta) \sum_{j=2}^{k+1} |\xi|^j\mathcal D_j(\eta) - (N_k(\xi)-I) |\xi|^{k+1}\mathcal D_{k+1}(\eta)\\
    &= \mathcal O( |\xi|^{k+2}).
  \end{align*}
  Using $N_k(\xi)-I=\mathcal O( |\xi| )$ we see that for $|\xi|\leq c_k$, $c_k$ sufficiently
  small, the matrix $N_k(\xi)$ is invertible and 
  \begin{equation}
    N_k^{-1}(\xi)B^{(0)}(\xi)N_k(\xi)=\sum_{j=1}^{k+1} |\xi|^j\mathcal D_j(\eta)+ \mathcal O( |\xi|^{k+2}).
  \end{equation}
  Thus, the entries of $\mathcal D_j(\eta)$ contain the asymptotic expansion
  of the eigenvalues, while the rows of $N_k(\xi)R(\eta)$ (and $L(\eta)N_k^{-1}(\xi)$) give 
  asymptotic expansions of the right (and left) eigenvectors of $B(\xi)$. 

  Furthermore, the construction implies that all occurring matrices are smooth functions of $\eta\in\S^1$.

  \medskip\noindent
  {\it Step 3.} We calculate the first terms explicitly.  For this we need
  the diagonal entries of the matrix $b_*\otimes\overline{{}_*b}$. Therefore, we determine
  the left and right eigenvectors of $B_1(\eta)$. If we assume that the direction $\eta$
  is non-degenerate we get for $e_0(\eta)=(r_1^+,r_1^-,r_2^+,r_2^-,r_0)^T$ and ${}_0e(\eta)=(\ell_1^+,\ell_1^-,\ell_2^+,\ell_2^-,\ell_0)^T$ 
  the equations
  \begin{subequations}
  \begin{align}
    \pm r_j^\pm \omega_j +\i\gamma a_j r_0 &=0,  & \pm\overline{\ell_j^\pm} \omega_j-\frac{\i\gamma}2 a_j\overline{\ell_0}&=0,\\
    a_1r_1^++a_2r_2^++a_1r_1^-+a_2r_2^-&=0,\qquad & a_1\overline{\ell_1^+}+a_2\overline{\ell_2^+}+a_1\overline{\ell_1^-}+a_2\overline{\ell_2^-}&=0,
  \end{align}
  \end{subequations}
  together with the normalisation condition
  \begin{equation}
      r_1^+\overline{\ell_1^+}+\cdots r_2^-\overline{\ell_2^-}+r_0\overline{\ell_0}=1.
  \end{equation}
  The first equations imply the representation 
  \begin{align}
    \pm r_j^\pm(\eta) &= -\frac{\i\gamma a_j(\eta)}{\omega_j(\eta)}\, r_0(\eta),\qquad \pm\overline{\ell_j^\pm(\eta)}=\frac{\i\gamma a_j(\eta)}{2\omega_j(\eta)}\, \ell_0(\eta),
  \end{align}
  the second line of equations follows from the first,
  while the normalisation condition yields
  \begin{equation}
    b_0(\eta)=r_0(\eta)\,\overline{\ell_0}(\eta) = \left(1+\frac{\gamma^2a_1^2(\eta)}{\varkappa_1(\eta)}+\frac{\gamma^2a_2^2(\eta)}{\varkappa_2(\eta)}\right)^{-1} \neq0.
  \end{equation}
  To calculate the eigenvectors we can further require $r_0(\eta)=\ell_0(\eta)=\sqrt{b_0(\eta)}>0$.

  Similarly, we obtain for the eigenvectors $e_k^+(\eta)=(r_1^+,\ldots,r_2^-,r_0)^T$ and 
  ${}_ke^+(\eta)=(\ell_1^+,\ldots,\ell_2^-,\ell_0)^T$ the equations (we use the same notation as above 
  in the hope that this will not lead to confusion here)
  \begin{align} \label{eq:eigenvect2}
    \pm r_j^\pm \omega_j +\i\gamma a_j r_0 &=\pm\tilde\nu_kr_j^\pm,  & \pm\overline{\ell_j^\pm} \omega_j-\frac{\i\gamma}2 a_j\overline{\ell_0}&=\pm\tilde\nu_k(\eta)\overline{\ell_j^\pm}
  \end{align}
  together with the normalisation condition. Thus for parabolic directions we get
  \begin{align}
      \pm r_j^\pm(\eta) &= \frac{\i\gamma a_j(\eta)}{\tilde\nu_k(\eta)-\omega_j(\eta)}\, r_0(\eta),\qquad \pm\overline{\ell_j^\pm(\eta)}=-\frac{\i\gamma a_j(\eta)}{2(\tilde\nu_k(\eta)-\omega_j(\eta))}\, \ell_0(\eta), 
  \end{align}
  and hence 
  \begin{align}
    b_k^+(\eta)=r_0(\eta)\,\overline{\ell_0}(\eta)&=\bigg(1+\sum_{j=1,2} \frac{\gamma^2a_j^2(\eta)}{2(\tilde\nu_k(\eta)-\omega_j(\eta))^2}
      +\sum_{j=1,2}\frac{\gamma^2a_j^2(\eta)}{2(\tilde\nu_k(\eta)+\omega_j(\eta))^2} \bigg)^{-1}\notag\\
    &=\left(1+ \gamma^2a_1^2(\eta)\frac{\tilde\nu_k^2(\eta)+\varkappa_1(\eta)}{(\tilde\nu_k^2(\eta)-\varkappa_1(\eta))^2}+\gamma^2a_2^2(\eta)\frac{\tilde\nu_k^2(\eta)+\varkappa_2(\eta)}{(\tilde\nu_k^2(\eta)-\varkappa_2(\eta))^2}
      \right)^{-1}.
  \end{align}
  For $e_k^-(\eta)$ and ${}_ke^-(\eta)$ we have to replace $\tilde\nu_k(\eta)$ by $-\tilde\nu_k(\eta)$ and obtain
  $b_k^-(\eta)=b^+_k(\eta)=b_k(\eta)$.

  If the direction $\eta$ is non-degenerate and hyperbolic with respect to the index $j_0$, the 
  entries $r_{j_0}^\pm$ and $\ell_{j_0}^\pm$ are undetermined by \eqref{eq:eigenvect2}, while the other entries 
  of the vectors are zero. Together with the normalisation condition this determines
  the eigenvectors and gives $b_{j_0}(\eta)=0$. It remains 
  to consider degenerate directions.
  Then we have $\tilde\nu_1=\omega_1$ and $\tilde\nu_2>\omega_1$
  such that for $k=1$ we have $\ell_0=r_0=0$, $r_1^\pm$ and $\ell_1^\pm$ are non-zero while 
  $r_2^\pm=\ell_2^\pm=0$, especially $b_1(\eta)=0$. For $k=2$ we get from the above expression
  for $b_2(\eta)=(2+2\varkappa^2/\gamma^2)^{-1}>0$.
\end{proof}

\noindent
{\bf Remark.} {\sl 1.} Note that for all non-degenerate hyperbolic directions $\bar\eta\in\S^1$ with respect to the index $1$ the limit
\begin{equation}\label{eq:ab-quot}
   \lim_{\eta\to\bar\eta} {a_1^2(\eta)}{b_1^{-1}(\eta)} = 
   \frac2{\gamma^2} 
   \varkappa_1(\bar\eta)\left(1-\frac{\gamma^2}{\varkappa_1(\bar\eta)-\varkappa_2(\bar\eta)}\right)^2
\end{equation}
is taken and non-zero, while for hyperbolic directions with respect to the index $2$ the corresponding
 limit is non-zero only if  $\gamma^2\ne\varkappa_2(\bar\eta)-\varkappa_1(\bar\eta)$, i.e. if the direction
 is  not $\gamma$-degenerate.
Near $\gamma$-degenerate directions Step 1 of the previous proof is still valid. Similar to Step 2 we can diagonalise to a $(2,2,1)$ block structure. The eigenvalues of these blocks can be calculated explicitely.

{\noindent\sl 2.} For degenerate directions we obtain similarly
\begin{equation}
  \lim_{\eta\to\bar\eta} \frac{b_1(\eta)}{(\varkappa_1(\eta)-\varkappa_2(\eta))^2}=\frac{a_1^2(\bar\eta)a_2^2(\bar\eta)}{2\gamma^2\varkappa_1(\bar\eta)}
\end{equation}
and $b_1(\eta)$ vanishes to the double contact order.

\subsection{Asymptotic expansion of the eigenvalues as $|\xi|\to\infty$}\label{sec2.2}

If we consider large frequencies the second order part $B_2(\xi)$ dominates
$B_1(\xi)$. This makes it necessary to apply a different two-step diagonalisation 
scheme. We follow partly ideas from \cite{ReiWang}, \cite{Wang}, \cite{Wang2}
adapted to our special situation.
\begin{prop}\label{prop:2.3}
   As $|\xi|\to\infty$ the eigenvalues of the matrix $B(\xi)$ behave as
  \begin{subequations}
    \begin{align}
      \nu_0(\xi)&=\i\kappa|\xi|^2-\frac{\i\gamma^2}\kappa+\mathcal O( |\xi|^{-1}),\\
      \nu_j^\pm(\xi)&=\pm|\xi|\omega_j(\eta)+\frac{\i\gamma^2}{2\kappa}a_j^2(\eta)+\mathcal O( |\xi|^{-1}).
    \end{align}
  \end{subequations}
  for all non-degenerate directions $\xi/|\xi|\in\S^1$.
\end{prop}
\noindent
{\bf Remark.} {\sl 1.} Note that, while in hyperbolic directions we always have 
$\nu_{j_0}^\pm(\xi)= \pm|\xi|\omega_{j_0}(\eta)$ for one index $j_0$, in degenerate hyperbolic directions
all $a_j(\eta)$ may be non-zero. Hence the statement of the above theorem cannot
be valid in such directions in general.

\noindent
{\sl 2.} Degenerate directions play for large frequencies a similar r\^ole as $\gamma$-degenerate directions play for small frequencies. 

\begin{proof} The proof will be decomposed into several steps. In a first
  step we use the main part $B_2(\xi)=\i\kappa|\xi|^2\diag(0,0,0,0,1)$ to block-diagonalise
  $B(\xi)$. In a second step we diagonalise the upper $4\times 4$ block for all 
  non-degenerate directions. 

  \medskip\noindent
  {\it Step 1.} For a matrix $B\in\C^{5\times5}$ we denote by $\bdiag_{4,1} B$
  the block diagonal of $B$ consisting of the upper $4\times4$ block and the lower
  corner entry. We construct a diagonalisation scheme to block-diagonalise
  $B(\xi)$ as $|\xi|\to\infty$. 

  We set $R^{(-1)}(\xi)=B(\xi)-B_2(\xi)=B_1(\xi)$ and $\mathcal B_{-2}(\xi)=B_2(\xi)$ 
  and construct recursively a diagonaliser modulo the upper $4\times4$ block,
  \begin{equation}
    M_k(\xi)=I+\sum_{j=1}^k |\xi|^{-j} M^{(j)}(\eta).
  \end{equation}
  Again we denote by $\tilde R^{(k)}(\eta)$ the $(-k)$-homogeneous part of 
  $R^{(k)}(\xi)$ (which exists because it exists for $R^{(-1)}(\xi)$ and the existence
  is transfered  by the construction). Then we introduce the recursive scheme
  \begin{align}
    \mathcal B_{k-2}(\eta)&=\bdiag_{4,1} \tilde R^{(k-2)}(\eta),\\
    M^{(k)}(\eta) &= \frac\i{\kappa}
    \begin{pmatrix}
       &&& \tilde R^{(k-2)}_{15}(\eta) \\
       &0&& \vdots \\
       &&& \tilde R^{(k-2)}_{45}(\eta) \\
       -\tilde R^{(k-2)}_{51}(\eta)&\cdots& -\tilde R^{(k-2)}_{54}(\eta) & 0 
    \end{pmatrix}
  \end{align}
   for $k=1,2,\ldots$, such that the commutator relation 
  \begin{equation}
       [\mathcal B_{-2}(\eta),M^{(k)}(\eta)] 
     = \mathcal B_{k-2}(\eta)-\tilde R^{(k-2)}(\eta),
  \end{equation}
  holds. Thus it follows  
  \begin{align}
    R^{(k-1)}(\xi)&=B(\xi)M_k(\xi)-M_k(\xi)\sum_{j=-2}^{k-2} |\xi|^{-j} \mathcal B_j(\eta)=\mathcal O( |\xi|^{1-k})
  \end{align}
  and using that $M_k(\xi)$ is invertible for $|\xi|\geq C_k$, $C_k$ sufficiently large,
  we obtain the block diagonalisation
  \begin{equation}
    M_k^{-1}(\xi)B(\xi)M_k(\xi)=\sum_{j=-2}^{k-2} |\xi|^{-j}\mathcal B_j(\eta)+\mathcal O( |\xi|^{1-k}),
  \end{equation}
  where $\mathcal B_j(\eta)=\bdiag_{4,1}\mathcal B_j(\eta)$ is $(4,1)$-block diagonal.
  
  \medskip\noindent
  {\it Step 2.} By Step 1 we constructed $M_k(\xi)$ such that $M_k^{-1}(\xi)B(\xi)M_k(\xi)$
  is $(4,1)$-block diagonal modulo $\mathcal O( |\xi|^{1-k})$. The upper $4\times4$ block
  has already diagonal main part $|\xi|^{-1}\mathcal D_{-1}(\eta)=B_1(\xi)$. If the direction
  $\eta=\xi/|\xi|$ is non-degenerate, the diagonal entries $\pm\omega_1(\eta)$ and $\pm\omega_2(\eta)$ are
  mutually distinct and thus we can apply the standard diagonalisation 
  procedure (cf. proof of Proposition~\ref{prop:2.2}) in the corresponding
  subspace. This does not alter the lower corner entry and gives only  
  combinations of the entries of the last column and of the last row (without
  changing their asymptotics).
  
  Thus we can construct a matrix $N_{k-1}(\xi)=I+\sum_{j=1}^{k-1} |\xi|^{-j}N^{(j)}(\xi)$, which is 
  invertible for $|\xi|>\tilde C_{k-1}$, $\tilde C_{k-1}$ sufficiently large, such that
  \begin{equation}
    N_{k-1}^{-1}(\xi)M_k^{-1}(\xi)B(\xi)M_k(\xi)N_{k-1}(\xi) = |\xi|^2\mathcal B_{-2}(\eta) 
    + \sum_{j=-1}^{k-2} |\xi|^{-j}\mathcal D_j(\eta)
    + \mathcal O( |\xi|^{1-k})
  \end{equation}
  is diagonal modulo $\mathcal O( |\xi|^{1-k})$.
  
  \medskip\noindent
  {\it Step 3.} We give the first matrices explicitly. Following Step 1 we get
  \begin{equation}
    M^{(1)}(\eta)=
    \begin{pmatrix}
      &&&& \frac{\gamma a_1(\eta)}{\kappa}\\  
      &&&& \frac{\gamma a_2(\eta)}{\kappa}\\
      &&&& \frac{\gamma a_1(\eta)}{\kappa}\\  
      &&&& \frac{\gamma a_2(\eta)}{\kappa}\\
      \frac{\gamma a_1(\eta)}{2\kappa}&\frac{\gamma a_2(\eta)}{2\kappa}&\frac{\gamma a_1(\eta)}{2\kappa}&\frac{\gamma a_2(\eta)}{2\kappa}&0
    \end{pmatrix}
  \end{equation}
  together with $\mathcal B_{-1}(\eta)=\diag(\omega_1(\eta),\omega_2(\eta),-\omega_1(\eta),-\omega_2(\eta),0)$ and 
  \begin{align}
    \tilde R^{(0)}(\eta) &= B_1(\eta)M^{(1)}(\eta) - M^{(1)}(\eta)\mathcal B_{-1}(\eta)\notag\\
    &=
    \begin{pmatrix}
      \i\frac{\gamma^2a_1^2(\eta)}{2\kappa} &  \i\frac{\gamma^2a_1(\eta)a_2(\eta)}{2\kappa} & \cdots & \frac{\gamma a_1(\eta)\omega_1(\eta)}{2\kappa}\\
      \i\frac{\gamma^2a_1(\eta)a_2(\eta)}{2\kappa} &  \i\frac{\gamma^2a_2^2(\eta)}{2\kappa} & \cdots & \frac{\gamma a_2(\eta)\omega_2(\eta)}{2\kappa}\\
      \vdots & \vdots & \ddots & \vdots \\
      \frac{\gamma a_1(\eta)\omega_1(\eta)}{2\kappa} &   \frac{\gamma a_2(\eta)\omega_2(\eta)}{2\kappa} & \cdots & -\i\frac{\gamma^2}{\kappa}
    \end{pmatrix},\\
    \mathcal B_0(\eta)&=\bdiag_{4,1}    \tilde R^{(0)}(\eta)
  \end{align}
  in the first diagonalisation step. Applying a second step alters only the last row and
  column to $\mathcal O( |\xi|^{-1})$. Following Step 2 we diagonalise the upper $4\times4$ block to
  $|\xi|\mathcal B_{-1}(\eta)+\mathcal B_{0}(\eta)+\mathcal O( |\xi|^{-1})$ modulo $\mathcal O( |\xi|^{-1})$
  and the statement is proven.
\end{proof}

\noindent{\bf Remark.} If the direction $\eta$ is degenerate, i.e. $\omega_1(\eta)=\omega_2(\eta)$,
we can block-diagonalise in Step 2 to $(2,2,1)$-block form. To diagonalise further we have to
know that the $0$-homogeneous part of these $2\times2$-blocks has distinct eigenvalues. 

One possible treatment of degenerate directions is given in the following proposition. Note that a corresponding statement can be obtained for $\gamma$-degenerate directions as $|\xi|\to0$.

\begin{prop}\label{prop:2.3a}
Let $\bar\eta$ be an isolated degenerate direction. Then the corresponding hyperbolic eigenvalue satisfies 
in a small conical neighbourhood of $\bar\eta$
\begin{multline}\label{eq:degEig}
  \nu_{j_0}^-(\xi)=\frac{\omega_1(\xi)+\omega_2(\xi)}2+\frac{\i\gamma^2}{4\kappa}\\-
	\sqrt{\frac{(\omega_1(\xi)-\omega_2(\xi))^2}4-\frac{\gamma^4}{16\kappa^2}+
\frac{\i\gamma^2(\omega_1(\xi)-\omega_2(\xi))(a_1^2(\eta)-a_2^2(\eta))}{4\kappa}}\\+\mathcal O(|\xi|^{-1}).
\end{multline}
\end{prop}
\begin{proof}
 We follow the treatment of the previous proof to $(2,2,1)$-block-diagonalise $B(\xi)$ modulo
 $|\xi|^{-1}$. Now, we consider one of its $2\times2$-blocks.
 (We use a similar notation as before in the hope that  it will not lead to any confusion.) Such a block is given by
 \begin{equation}
    \mathcal B(\xi)=|\xi| \mathcal B_{-1}(\eta)+\mathcal B_0(\eta)+\mathcal O(|\xi|^{-1}),
 \end{equation}
 where
 \begin{align}
   \mathcal B_{-1}(\eta)&=\diag\big(\omega_1(\eta),\omega_2(\eta)\big),\\
   \mathcal B_{0}(\eta)&=\frac{\i\gamma^2}{2\kappa} \begin{pmatrix} a_1^2(\eta)&a_1(\eta)a_2(\eta)\\
   a_1(\eta)a_2(\eta)&a_2^2(\eta) \end{pmatrix}.
 \end{align}
 In the direction $\bar\eta$ both diagonal entries of $\mathcal B_{-1}$ coincide.
% and the spectrum of $|\xi| \mathcal B_{-1}(\bar\eta)+\mathcal B_0(\bar\eta)$ consists of $\omega_1(\xi)$ and
% $\omega_1(\xi)+\frac{\i\gamma^2}{2\kappa}$. 
 In a small conical neighbourhood we denote the eigenvalues of $ |\xi| \mathcal B_{-1}(\eta)+\mathcal B_0(\eta)$
 as $\delta_+(\xi)$ and $\delta_-(\xi)$. A simple calculation yields
 \begin{multline}\label{eq:delta-}
   \delta_\pm (\xi) = \frac{\omega_1(\xi)+\omega_2(\xi)}2+\frac{\i\gamma^2}{4\kappa}\\  \pm
	\sqrt{\frac{(\omega_1(\xi)-\omega_2(\xi))^2}4-\frac{\gamma^4}{16\kappa^2}+
\frac{\i\gamma^2(\omega_1(\xi)-\omega_2(\xi))(a_1^2(\eta)-a_2^2(\eta))}{4\kappa}}
 \end{multline}
 with $\delta_-(\bar\xi)=\omega_1(\bar\xi)$ and $\delta_+(\bar\xi)=\omega_1(\bar\xi)+\frac{\i\gamma^2}{2\kappa}$.
The hyperbolic eigenvalue corresponds to $\delta_-(\xi)$. These eigenvalues are distinct in a sufficiently small neighbourhood of $\bar\eta$ (and may coincide only if $a_1(\eta)=a_2(\eta)$ and 
$(\omega_1(\xi)-\omega_2(\xi))^2=\gamma^4/(4\kappa^2)$, which gives eventually two parabolic directions). 

Hence the perturbation theory of matrices implies $\nu_{j_0}^+(\xi)=\delta_-(\xi)+\mathcal O(|\xi|^{-1})$
in a sufficiently small neighbourhood of $\bar\eta$ and the statement is proven.
\end{proof}

\noindent
{\bf Remark.} Note, that
\begin{equation}
\lim_{\eta\to\bar\eta} \frac{\omega_1(\xi)-\delta_-(\xi)}{\omega_1(\xi)-\omega_2(\xi)}= a_1^2(\bar\eta)
\end{equation}
for all fixed $|\xi|$, which coincides with the result \eqref{eq:2.13} of Proposition~\ref{prop:2.0a} for the eigenvalue $\nu(\xi)$. 

\subsection{Collecting the results}

The asymptotic expansions from Propositions~\ref{prop:2.2} and Proposition~\ref{prop:2.3} imply 
estimates for eigenvalues of $B(\xi)$ and the proofs give representations of corresponding eigenvectors.
For the application of multiplier theorems and the proof of $L^p$--$L^q$ decay estimates it is essential to provide also {\em estimates for derivatives} of them.

Assume that the eigenvalues under consideration are simple. From the asymptotic expansions we know
that this is the case for small frequencies and also for large frequencies.
For the middle part it will be sufficient to know that the hyperbolic eigenvalues are separated, which follows
for sufficiently small conical neighbourhoods of these directions.

In a first step we consider derivatives of the eigenvalues. Differentiating the characteristic polynomial
\begin{equation}
  0=\det(\nu(\xi)I-B(\xi))=\sum_{k=0}^5 I_k(\xi) \nu(\xi)^k
\end{equation}
with respect to $\xi$ yields by Leibniz formula
\begin{equation}
 0 = \sum_{k=0}^5 \sum_{\beta\le\alpha} \binom\alpha\beta 
 \big(\D_\xi^{\alpha-\beta} I_k(\xi)\big) \big(\D_\xi^\beta \nu(\xi)^k\big)
\end{equation}
for all multi-indices $\alpha\in\mathbb N_0^2$. Thus we can express the highest derivative of $\nu(\xi)^k$ in terms of lower ones and hence Fa\`a di Bruno's formula (see e.g. \cite{FaaDiBr}) yields an expression
\begin{align}
 & \big(\D^\alpha\nu(\xi)\big)\sum_{k=1}^5 kI_k(\xi) \nu(\xi)^{k-1} = \sum_{k=0}^5 \sum_{\beta<\alpha}
 C_{k,\alpha,\beta} \big(\D^{\alpha-\beta}I_k(\xi)\big)\big(\D^\beta \nu(\xi)\big) 
\end{align}
with certain constants $C_{k,\alpha,\beta}$.
Because the eigenvalue has multiplicity one, the sum on the left-hand 
side is nonzero and therefore we can calculate the derivatives of $\nu(\xi)$ by this expression. 
Furthermore, it follows that for small and large frequencies the derivatives of the eigenvalue have
full asymptotic expansions and thus we are allowed to differentiate the asymptotic expansions term by term.

It remains to consider the corresponding eigenprojections. Recall that if $\nu(\xi)$ is a eigenvalue of multiplicity
one and $r(\xi)$ and $l(\xi)$ are corresponding right and left eigenvectors, the corresponding eigenprojection is
given by the dyadic product $P_\nu(\xi)=l(\xi)\otimes r(\xi)$. Thus the constructed diagonaliser matrices imply asymptotic
expansions of these operators. Again we are only interested whether the derivatives of these eigenprojections
also possess asymptotic expansions (in order to see whether it is allowed to differentiate term by term).

For this we use the representation 
\begin{equation}
  P_\nu(\xi)=\prod_{\tilde\nu \in  \spec{B(\xi)}\setminus\{\nu\} } (\tilde\nu(\xi)I-B(\xi))(\tilde\nu(\xi)-\nu(\xi))^{-1}
\end{equation}
given e.g. in \cite{Doll}, \cite{Liess}. All terms on the right-hand side have full asymptotic expansions
as $|\xi|\to0$ and $|\xi|\to\infty$ together with all of their derivatives. Differentiating with respect to $\xi$ yields the
same result for the eigenprojection. Thus we obtain

\begin{prop}
  The asymptotic expansions from Proposition~\ref{prop:2.2} and Proposition~\ref{prop:2.3} may be differentiated
  term by term to get asymptotic expansions for the derivatives of the eigenvalues. Furthermore, the same holds true
  for the corresponding eigenprojections.
\end{prop}

From Proposition~\ref{prop:2.0} we know that
an eigenvalue $\nu(\xi)$ of the matrix $B(\xi)$ is real if and only if $\eta=\xi/|\xi|$ is hyperbolic.
We want to combine this information with the asymptotic expansions of Proposition~\ref{prop:2.2}
and Proposition~\ref{prop:2.3} and derive some estimates for the {\em behaviour of the imaginary part}.

\begin{prop}\label{prop:2.4} Let $c>0$ be a given constant.
  \begin{enumerate}
  \item Let $\eta=\xi/|\xi|\in\S^1$ be parabolic. Then the eigenvalues of $B(\xi)$ satisfy
    \begin{equation}
      \Im \nu(\xi) \geq C_\eta  >0,\qquad |\xi|\geq c,
    \end{equation}
    with a constant $C_\eta$ depending on the direction $\eta$ and $c$. Furthermore,
    \begin{equation}
      \Im \nu(\xi)\sim b(\eta)|\xi|^2,\qquad |\xi|\leq c,
    \end{equation}
    where $b(\eta)$ is one of the functions from Proposition~\ref{prop:2.2}.
  \item Let $\bar\eta$ be non-degenerate and hyperbolic with respect to the index $1$. 
    Then $\nu_{1}^\pm(|\xi|\bar\eta)=\pm|\xi|\omega_1(\bar\eta)$ 
    and $\nu_0(\xi)$ and $\nu_2^\pm(\xi)$ satisfy the statement of point 1. Furthermore,
    \begin{equation}\label{eq:zeroImEig1}
      \Im \nu_1^\pm(\xi) \sim  a_1^2(\eta),\qquad |\xi|\geq c,\quad |\eta-\bar\eta|\ll1,
    \end{equation}
    and
    \begin{equation}\label{eq:zeroImEig2}
      \Im \nu_1^\pm(\xi) \sim |\xi|^2a_1^2(\eta), \qquad |\xi|\leq c,\quad |\eta-\bar\eta|\ll1.  
    \end{equation}
  \end{enumerate}
\end{prop}
\begin{proof}
  The first point follows directly from the asymptotic expansions, we concentrate on the second
  one. We know that the hyperbolic eigenvalues $\nu_1^\pm(\xi)$ satisfy by Proposition~\ref{prop:2.0.1}
  \begin{equation}
     \Im \nu_1^\pm(\xi) = a_1^2(\eta) \mathrm N_1^\pm(\xi)
  \end{equation}
  with a smooth non-vanishing function $\mathrm N_1^\pm(\xi)$ defined in a 
  neighbourhood of $\bar\eta$. By   
  Proposition~\ref{prop:2.2} and \ref{prop:2.3} we see that $\mathrm N_1^\pm(\xi)$ also has
  full asymptotic expansions and thus
  \begin{subequations}
  \begin{align}
      \mathrm N_1^\pm (\xi)&=\frac{\i\gamma^2}{2\kappa}+\mathcal O(|\xi|^{-1}),\qquad |\xi|\to\infty\\
      \mathrm N_1^\pm (\xi)&=\i\kappa|\xi|^2\frac{b_1(\eta)}{a_1^2(\eta)}+\mathcal O(|\xi|^3),
      \qquad |\xi|\to0.
  \end{align}
  \end{subequations}
  Together with \eqref{eq:ab-quot}  we get upper and lower bounds on $\mathrm N_1^\pm(\xi)$
  and the desired statement follows.\end{proof}

\noindent
{\bf Remark.} A similar reasoning allows to replace the hyperbolic eigenvalue near degenerate
directions by the model expression obtained in Proposition~\ref{prop:2.3a}, thus 
$\nu(\xi)\sim\delta_-(\xi)$ uniformly in a sufficiently small conical neighbourhood of $\bar\eta$.

\section{Decay estimates for solutions}
\label{sec:3}

Our strategy to give decay estimates for solutions to the thermo-elastic system \eqref{eq:System1} is to 
micro-localise them. In principle we have to distinguish four different cases. On the one hand
we differentiate between small and large frequencies, on the other hand between hyperbolic directions and parabolic ones. 

We distinguish between two cases depending on the vanishing order of the coupling functions. If the
coupling functions vanish to first order at hyperbolic directions only, we rely on simple multiplier estimates. Later on we discuss coupling functions with higher vanishing order, where the decay rates are obtained by tools closely related to the treatment of the elasticity equation.

\subsection{Coupling functions with simple zeros}
\label{sec3.1}

In a first step we consider the first order system
\begin{equation}\label{eq:System3}
  \D_t V = B(\D) V,\qquad V(0,\cdot)=V_0
\end{equation}
to Cauchy data $V_0\in\mathcal S(\R^2,\C^5)$. For a cut-off function $\chi\in C^\infty(\R_+)$ with
$\chi(s)=0$, $s\leq \epsilon$, and $\chi(s)=1$, $s\geq 2\epsilon$, we consider 
\begin{align}
  P_{par}(\eta)&= \prod_{\text{$\bar\eta$ hyperbolic}} \chi(|\eta-\bar\eta|),\qquad &
  P_{hyp}(\eta)&= 1- P_{par}(\eta).
\end{align}
Then $P_{hyp}(\D)$ localises in a conical neighbourhood of the set of hyperbolic directions,
while $P_{par}(\D)$ localises to a compact set of parabolic directions.  

The asymptotic formulae and representations for the characteristic roots of the full symbol $B(\xi)$ 
allow us to proof decay estimates for the solutions.

\begin{thm}\label{thm3.1}
  Assume that (A1) to (A4) are satisfied and the coupling functions vanish in hyperbolic
  directions to first order.
  
  Then the solution $V(t,x)$ to \eqref{eq:System3} satisfies the following {\sl a-priori} estimates:
  \begin{subequations}
    \begin{align}
      \| \chi(\D)P_{par}(\D) V(t,\cdot)\|_q &\lesssim e^{-Ct} \|V_0\|_{p,r}\\
      \| (1-\chi(\D))P_{par}(\D) V(t,\cdot)\|_q &\lesssim (1+t)^{-(\frac1p-\frac1q)} \|V_0\|_{p}\\
      \| \chi(\D)P_{hyp}(\D) V(t,\cdot)\|_q &\lesssim  (1+t)^{-\frac12(\frac1p-\frac1q)} \|V_0\|_{p,r}\\
      \| (1-\chi(\D))P_{hyp}(\D) V(t,\cdot)\|_q &\lesssim  (1+t)^{-\frac12(\frac1p-\frac1q)} \|V_0\|_{p}
    \end{align}
  \end{subequations}
  for dual indices $p\in(1,2]$, $pq=p+q$, and with Sobolev regularity $r>2(1/p-1/q)$.
\end{thm}

\noindent{\bf Remark.} If $B(\xi)$ is diagonalisable for $\xi\neq0$ (which is valid e.g. if $B(\xi)$ has no double eigenvalues 
for $\xi\neq0$) we can make the result even more precise.
To each eigenvalue $\nu(\xi)\in\spec B(\xi)$ we have corresponding left and right eigenvectors and associated to them
the eigenprojection $P_\nu(\D)$ such that 
\begin{equation}
  P_\nu(\D)V(t,\cdot)=e^{\i t \nu(\D)} P_\nu(\D) V_0.
\end{equation}
Thus, we can single out the influence of one eigenvalue in this way. Note, that this is only of interest
in the neighbourhood of hyperbolic directions and for the corresponding eigenvalue and there the assumption
of diagonalisability of $B(\xi)$ may be skipped (real eigenvalues are always simple, for small $|\xi|$ diagonalisation
works, for large $|\xi|$ everything goes well under the assumption of non-degeneracy and on the middle part
we make the neighbourhood small enough to exclude possible multiplicities).

\begin{proof} We decompose the proof into four parts corresponding to the 
four estimates. The micro-localised estimates are merely standard multiplier 
estimates. We do {\em not} use stationary phase method.
  
\medskip\noindent
{\it Step 1. Parabolic directions, large frequencies.} In this case we have 
uniformly in $\xi\in\supp \chi P_{par}$, the estimate $\Im\nu(\xi)\geq C'>0$. Taking $0<C<C'$
we obtain for these $\xi$
\begin{equation*}
   \Im B(\xi) = \frac{B(\xi)-B^*(\xi)}{2\i} \geq CI
\end{equation*}
in the sense of self-adjoint operators and the estimate follows from the $L^2$--$L^2$ estimate
\begin{align*}
 \frac{\d}{\d t}& \| \chi(\D)P_{par}(\D) V(t,x) \|^2_2 
 =  \frac{\d}{\d t} \| \chi(\xi)P_{par}(\xi) \hat V(t,\xi) \|^2_2 \\
 &= 2\Re  \big( \chi(\xi)P_{par}(\xi)\hat V(t,\xi),\chi(\xi)P_{par}(\xi) \partial_t \hat V(t,\xi) \big)\\
 &= -2\Im  \big( \chi(\xi)P_{par}(\xi)\hat V(t,\xi),\chi(\xi)P_{par}(\xi) B(\xi) \hat V(t,\xi) \big)\\
 &\leq -2 C  \| \chi(\xi)P_{par}(\xi) \hat V(t,\xi) \|^2=-2C \| \chi(\D)P_{par}(\D) V(t,x) \|^2_2, 
\end{align*}
viewed as $H^s$--$H^s$ estimate and combined with Sobolev embedding.

\medskip\noindent
{\it Step 2. Parabolic directions, small frequencies.} We know from 
Proposition~\ref{prop:2.2} that in this case the matrix $B(\xi)$ has only simple 
eigenvalues. We will make use of the representation of solutions as 
\begin{equation}\label{eq:SolRep}
   V(t,x) = \sum_{\nu(\xi)\in\spec B(\xi)} e^{\i t\nu(\D)} P_\nu(\D) V_0(x) 
\end{equation}
with corresponding eigenprojections (amplitudes) $P_\nu(\xi)$. The amplitudes are 
uniformly bounded on the set of all occurring $\xi$ and possess full asymptotic
expansions in $|\xi|$ as $\xi\to0$ together with their derivatives. 
Especially, by H\"ormander-Mikhlin multiplier theorem \cite[p. 96, Theorem 3]{Stein1} the operators $P_\nu(\D)$ are 
$L^p$-bounded for $1<p<\infty$.

It remains to consider the model multiplier $e^{\i t\nu(\xi)}$. From 
Proposition~\ref{prop:2.2} we know that
\begin{equation}
  |e^{\i t\nu(\xi)}|\lesssim  e^{-Ct|\xi|^2}, \qquad |\xi|\leq c,
\end{equation}
with suitable constants $c$ and $C$ and thus the $L^1$--$L^\infty$ estimate
\begin{align*}
  \|e^{\i t\nu(\D)} f\|_\infty &\leq \|e^{\i t\nu(\xi)}\hat f\|_1 \leq \|e^{\i t\nu(\xi)}\|_{L^1( \{ |\xi|\leq c \})} \|\hat f\|_\infty \\
  &\lesssim \|f\|_1 \int_0^c e^{-Ct |\xi|^2} |\xi|\d|\xi| \lesssim (1+t)^{-1} \;\|f\|_1
\end{align*}
holds for all $f\in L^1(\R^2)$ with $\supp\hat f\subseteq \{ |\xi|\leq c \}$. Riesz-Thorin 
interpolation \cite[Chapter 4.2]{BenSharp} with the obvious $L^2$--$L^2$ estimate gives the desired decay 
result.

\medskip\noindent
{\it Step 3. Hyperbolic directions, large frequencies.} We take the 
conical neighbourhoods small enough to exclude all multiplicities (related
to the eigenvalue which becomes real in the hyperbolic direction). Then 
similar to \eqref{eq:SolRep} the solution is represented as
\begin{equation}\label{eq:solRepHypFr}
  V(t,x)=e^{\i t\nu(\D)}P_\nu^+(\D)V_0(x)+e^{-\i t\nu(\D)}P_\nu^-(\D)V_0(x) +\tilde V(t,x),
\end{equation}
where $\tilde V(t,x)$ corresponds to the remaining parabolic eigenvalues of
$B(\xi)$ and satisfies the estimate from Step 1. Again we can use smoothness of
$P_\nu^\pm(\xi)$ together with the existence of a full asymptotic expansion as $|\xi|\to\infty$
to get $L^p$-boundedness of $P_\nu^\pm(\D)$ for $1<p<\infty$ from H\"ormander-Mikhlin 
multiplier theorem.

It remains to understand the model multiplier $e^{\pm\i t\nu(\xi)}$ for the hyperbolic 
eigenvalue related to the hyperbolic direction $\bar\eta$. 
Using the estimate from Proposition~\ref{prop:2.4} 
we conclude for $r>2$ 
\begin{align*}
  \| e^{\pm\i t\nu(\D)} f\|_\infty  &\leq \| e^{\pm\i t\nu(\xi)} \hat f\|_1 \leq \|e^{\pm\i t\nu(\xi)}|\xi|^{-r} \|_{L^1(S_1)} \|\,|\xi|^{r}\hat f\|_\infty \\
  &\lesssim \|\langle\D\rangle^r f\|_1 \; \int_c^\infty |\xi|^{1-r}\d |\xi|  \; \int_{-\epsilon}^\epsilon e^{-C_1 t\phi^2} \d\phi \\&\lesssim t^{-1/2}\; \|\langle\D\rangle^r f\|_1,\qquad t\geq1
\end{align*}
for all $f\in\langle\D\rangle^{-r}L^1(\R^2)$ with $\supp\hat f\subseteq S_1=\{ |\xi|\geq c,\; |\eta-\bar\eta|\leq\epsilon\} $. 
Riesz-Thorin interpolation with the $L^2$--$L^2$ estimate gives the 
desired decay result.

\medskip\noindent
{\it Step 4. Hyperbolic directions, small frequencies.}  Like for large hyperbolic 
frequencies we make use of the representation \eqref{eq:solRepHypFr} to separate hyperbolic
and parabolic influences. As in the previous cases the existence of full asymptotic
expansions imply that the projections $P_\nu^\pm(\D)$ are $L^p$-bounded for $1<p<\infty$. 

It remains to understand the model multiplier $e^{\pm\i t\nu(\xi)}$. Using Proposition~\ref{prop:2.4} 
we have $|e^{\pm\i t\nu(\xi)}|\lesssim e^{-C_2 t|\xi|^2\phi^2}$ such that after introducing polar co-ordinates
\begin{align*}
   \| e^{\pm\i t\nu(\D)} f\|_\infty  &\leq \| e^{\pm\i t\nu(\xi)} \hat f\|_1 \leq \|e^{\pm\i t\nu(\xi)} \|_{L^1(S_2)} \|f\|_1 \\
  &\lesssim \|f\|_1 \; \int_0^c \int_{-\epsilon}^\epsilon e^{-C_2 t\phi^2|\xi|^2} \d\phi |\xi|\d|\xi| \\
  &\lesssim t^{-1/2} \|f\|_1 \;\int_0^c \d |\xi| \lesssim t^{-1/2}\;\|f\|_1,\qquad t\geq1
\end{align*}
for all $f\in L^1(\R^2)$ with $\supp\hat f\subseteq S_2=\{ |\xi|\leq c,\;  |\eta-\bar\eta|\leq\epsilon\} $. 
Riesz-Thorin interpolation with the $L^2$--$L^2$ estimate gives the 
desired decay result.
\end{proof}

\noindent
{\bf Remark.} {\sl 1.}
In non-degenerate hyperbolic directions where the coupling function vanishes to order $\ell$ 
(cf. Example~\ref{expl:rhombic} case 3) we obtain by 
the same reasoning the weaker $L^p$--$L^q$ decay rate
\begin{equation}
    \| \chi(\D)P_{hyp}(\D) V(t,\cdot)\|_q \lesssim  (1+t)^{-\frac1{2\ell}(\frac1p-\frac1q)} \|V_0\|_{p,r}.
\end{equation}
It remains to understand whether this weaker decay rate is also sharp or whether an application
of stationary phase method may be used to improve this. We will discuss this in Section~\ref{app1}.\\
{\sl 2.} We can extend the estimate of  Theorem~\ref{thm3.1} to the limit case $p=1$, if we 
include the eigenprojections $P_\nu(\xi)$ into the considered model multiplier and use just their boundedness instead of H\"ormander-Mikhlin multiplier theorem. This is possible, because all the multiplier estimates were based on H\"older inequalities.

So far we understood properties of solutions to the transformed problem \eqref{eq:System3} for the 
vector-valued function $V(t,x)$ given by \eqref{eq:hatVdef} as
\begin{equation}
  V(t,x) =
  \begin{pmatrix}
    (\D_t+\mathcal D^{1/2}(\D)) U^{(0)}(t,x)\\
    (\D_t-\mathcal D^{1/2}(\D)) U^{(0)}(t,x)\\
    \theta(t,x)
  \end{pmatrix},
\end{equation}
where $U^{(0)}(t,x)=M^T(\D)U(t,x)$ is the elastic displacement after transformation with the diagonaliser
of the elastic operator. Because $M(\eta)$ is unitary and homogeneous of degree zero this diagonaliser is 
$L^p$-bounded for $1<p<\infty$ with bounded inverse. Thus we have 
\begin{align}
  \D_tU(t,x)&=M(\D)
  \begin{pmatrix}
    \frac12 & 0 & \frac12 & 0 & 0 \\
    0 & \frac12 & 0 & \frac12 & 0
  \end{pmatrix}
  V(t,x)\\
  \sqrt{A(\D)} U(t,x) &=M(\D)
  \begin{pmatrix}
    \frac12 & 0 & -\frac12 & 0 & 0 \\
    0 & \frac12 & 0 & -\frac12 & 0
  \end{pmatrix}
  V(t,x)
\end{align}
such that as corollary of Theorem~\ref{thm3.1} we obtain

\begin{cor}\label{cor3.1}
  Assume (A1) to (A4) and that the coupling functions vanish of first order in all hyperbolic directions.
 Then the solution $U(t,x)$ and $\theta(t,x)$ to \eqref{eq:System1} satisfy the  a-priori estimates
  \begin{multline}
    \|\D_t U(t,\cdot)\|_q+\|\sqrt{A(\D)} U(t,\cdot)\|_q + \|\theta(t,\cdot)\|_q \\
    \lesssim     (1+t)^{-\frac12(\frac1p-\frac1q)} \left(\|U_1\|_{p,r+1} + \|U_2\|_{p,r} + \|\theta_0\|_{p,r} \right)
  \end{multline}
  for dual indices $p\in(1,2]$, $pq=p+q$, and Sobolev regularity $r>2(1/p-1/q)$.
\end{cor}

\noindent
{\bf Remark.} Including the diagonaliser $M(\xi)$ into the model multipliers of Theorem~\ref{thm3.1}
allows to overcome the restriction on $p$ and to extend this statement up to $p=1$. We preferred this way of presenting the results because they allow to decouple both statements. Theorem~\ref{thm3.1} 
gives deeper insight into the asymptotic behaviour of solutions than Corollary~\ref{cor3.1} does.

\subsection{Coupling functions vanishing to higher order}\label{sec3.2}
\label{app1}
We have seen that under the assumption that the coupling functions $a_j(\eta)\in C^\infty(\S^1)$, $a_j(\eta)=\eta\cdot r_j(\eta)$ related to
the symbol of the elastic operator $A(\eta)$ have only zeros of first order, we can deduce $L^p$--$L^q$ decay estimates
without relying on the method of stationary phase. 

Now we want to discuss how to use the method of stationary phase to deduce decay estimates in the
remaining cases.  From Proposition~\ref{prop:2.0.1} we know that in hyperbolic directions
the imaginary part $\Im \nu_{j_0}^\pm(\xi)$ of the hyperbolic eigenvalues vanishes 
like the square of the corresponding coupling function $a_{j_0}^2(\eta)$, while the real part
$\Re \nu_{j_0}^\pm(\xi)$ is essentially described by $\pm\omega_{j_0}(\xi)$.  First, we make this more
precise and formulate an estimate for $\Re \nu_{j_0}(\xi)^\pm\mp\omega_{j_0}(\xi)$ and its derivatives.

\begin{prop}\label{prop3.3}
  Let $\bar\eta\in\S^1$ be hyperbolic with respect to the index $j_0$ and $a_{j_0}(\eta)$
  vanish to order $\ell$ in $\bar\eta$. Then there exists a conical neighbourhood of the direction $\bar\eta$
  such that for all $k=0,1,\ldots,2\ell-1$ the estimates
  \begin{subequations}
  \begin{align}
     \frac1{|\xi|}\left| \partial_\eta^k (\Re \nu_{j_0}^\pm(\xi)\mp\omega_{j_0}(\xi))\right|&\le c(|\xi|)  |\eta-\bar\eta|^{2\ell-k},\\
     \left| \partial_\eta^k \Im \nu_{j_0}^\pm(\xi)\right| &\le  c(|\xi|) |\eta-\bar\eta|^{2\ell-k},
  \end{align}
  \end{subequations}
  hold uniformly on it, where $c(|\xi|)\sim 1$ as $|\xi|\to\infty$ and $c(|\xi|)\sim|\xi|$ as $|\xi|\to0$.
\end{prop}

This estimate may be used to transfer the micro-localised decay estimate for the elasticity equation
to the thermo-elastic system on a sufficiently small conical neighbourhood of $\bar\eta$.  We need one further 
notion to prepare the main theorem of this section. Decay rates for solutions to the elasticity equation depend 
heavily on the {\em order of contact} between the sheets of the Fresnel surface and its tangents, cf.  \cite{Sug96} or 
point 2 of the concluding remarks on page~\pageref{conclrem}.

\begin{prop}\label{prop3.4}
For $\eta\in\S^1$ we denote by ${\bar\gamma}_j(\eta)$ the order of contact between the $j$-th sheet 
\begin{equation} 
S_j=\{\omega_j^{-1}(\eta)\eta\,|\,\eta\in\S^{1}\} 
\end{equation}
of the Fresnel curve and its tangent in the point $\omega_j^{-1}(\eta)\eta$.
Then
\begin{equation}
   \partial_\eta^k( \partial_\eta^2\omega_{j}(\eta)+\omega_j(\eta))=0, \qquad k=0,\ldots,{\bar\gamma}_j(\eta)-2,
\end{equation}
(if ${\bar\gamma}_j(\eta)>2$) and
\begin{equation}
 \partial_\eta^{{\bar\gamma}_j(\eta)}\omega_j(\eta)+\partial_\eta^{{\bar\gamma}_j(\eta)-2}\omega_j(\eta)\ne0.
\end{equation} 
\end{prop}
\begin{proof}
  Follows by straight-forward calculation. The curvature of the $j$-th Fresnel curve $S_j$ at
  the  point $\omega_j^{-1}(\eta)\eta$ factorises as $\omega_j(\eta)+\partial_\eta^2\omega_j(\eta)$
  and a smooth non-vanishing analytic function. 
\end{proof}

For the following we chose the conical neighbourhood of the hyperbolic direction
$\bar\eta$ (with respect to the index $j_0$) small enough in order that the curvature of $S_{j_0}$ vanishes at most in
$\bar\eta$. 

\begin{thm} \label{thm3.5}
Let $\bar\eta$ be hyperbolic with respect to the index $j_0$ and let $a_{j_0}(\eta)$ vanish in
$\bar\eta$ to order $\ell>1$.  Let us assume further that ${\bar\gamma}_{j_0}(\bar\eta)<2\ell$. 
Then 
\begin{equation}
  || e^{\pm\i t\nu_{j_0}(\D)} f ||_{q} \lesssim (1+t)^{-\frac1{{\bar\gamma}_{j_0}(\bar\eta)}(\frac1p-\frac1q)} ||f||_{p,r}
\end{equation}
for dual indices $p\in[1,2]$, $pq=p+q$, regularity $r>2(1/p-1/q)$ and 
for all $f$ with $\supp \hat f$ contained in a sufficiently small conical neighbourhood of $\bar\eta$. 
\end{thm}
\begin{proof}
We distinguish between two cases related to the Fourier support of $f$, for general $f$ we can use linearity. 
In both cases we apply the method of stationary phase. Key lemma will be the Lemma of van der Corput, cf. \cite[p.334]{Stein}, combined for 
large frequencies with a dyadic decomposition.  It suffices to consider $t\ge1$, the estimate 
for $t\le1$ follows from the uniform boundedness of the Fourier multiplier together with
Sobolev embedding using the regularity imposed on the data, $r>2(1/p-1/q)$.

In the following we skip the index $j_0$ of the eigenvalue. Thus $\omega(\eta)$ stands for $\omega_{j_0}(\eta)$
and ${\bar\gamma}(\eta)$ for ${\bar\gamma}_{j_0}(\eta)$ to shorten the notation.

\noindent
{\it Large and medium frequencies.} Assume that $\hat f$ is supported in a sufficiently small conical neighbourhood of $\bar\eta$ bounded away from zero, $|\xi|>1$. We will use a dyadic decomposition in the radial variable. 
Let for this $\chi\in C_0^\infty(\R)$ be chosen in such a way that $\chi(\R)=[0,1]$, $\supp\chi\subseteq[1/2,2]$ and $\sum_{j\in\mathbb Z} \chi(2^js)=1$   for all $s\in\R_+$. We set $\chi_j(s)=\chi(2^{-j}s)$ such that 
$\supp\chi_j\subseteq[2^{j-1},2^{j+1}]$. 

We follow the treatment of Brenner \cite{Bre75} and use Besov spaces. Due to our assumptions on the Fourier support 
of $f$ Besov norms are given by 
\begin{equation}
   ||f||_{B^r_{p,q}}=  \big\| 2^{jr}  ||\chi_j(|\D|)f ||_p \big\|_{\ell^q(\mathbb N_0)}
\end{equation}   
and corresponding Besov spaces are related to Sobolev and $L^p$-spaces by the embedding relations 
\begin{equation}
H^{p,r}\hookrightarrow B^r_{p,2},\qquad B^0_{q,2}\hookrightarrow L^q
\end{equation}
for $p\in(1,2]$ and $q\in[2,\infty)$. Thus it suffices to prove an $B^r_{p,2}\to B^0_{q,2}$ estimate
for dual $p$ and $q$. The case $p=q=2$, $r=0$ is trivial by Plancherel and the uniform boundedness of the multiplier,  
we concentrate on the $B^2_{1,2}\to  B^0_{\infty,2}$ estimate.
Therefore, we use
\begin{equation}\label{eq3.19}
  || e^{\pm\i t\nu(\D)}f||_{B^0_{\infty,2}}=\left(\sum_{j\in\mathbb Z}  ||  \chi_j(\D) e^{\pm\i t\nu(\D)}f||_\infty
  ^2\right)^{1/2}
  \le  \sup_j I_j(t) \;
  ||f||_{B^r_{1,2}},
 \end{equation}
where $I_j(t)$ are estimates of the dyadic components of the operator,
  \begin{equation}
    I_j(t)=\sup_{x\in\R^2} |\mathcal F^{-1} [ e^{\i t\nu(\xi)}\psi(\xi) \chi_j( |\xi| )|\xi|^{-r} ]|.
  \end{equation}
  Here $\psi(\xi)$ localises to a small conical neighbourhood of $\bar\eta$ bounded away from zero, $|\xi|>1$, 
  containing the support of $\hat f$.  We assume $|\psi(\xi)|+|\partial_\eta\psi(\xi)|+|\partial_{|\xi|}\psi(\xi)|\lesssim1$ and $\psi\equiv1$ on
  a smaller conic set around $\bar\eta$.
  We introduce polar co-ordinates $|\xi|$ and $\phi$ (with $\phi=0$ corresponding 
  to $\bar\eta$) and set $x=tz$.
  Using that the hyperbolic directions $\bar\eta\in\S^1$ are isolated we can assume that the 
  conical neighbourhood under consideration contains no further hyperbolic direction. For the 
  calculation of $I_j(t)$ we integrate over the interval $[-\epsilon,\epsilon]$ for $\phi$.  
  This yields for all $j\in\mathbb N_0$ 
  \begin{align*}
    I_j(t) &= \sup_{z\in\R^2} 
    	\left| \int_{2^{j-1}}^{2^{j+1}} \int_{-\epsilon}^\epsilon  e^{\i t( \Re \nu(\xi)+z\cdot\xi ) -t\Im \nu(\xi) }
	\psi(\xi)\chi_j( |\xi| )|\xi|^{1-r}\d\phi \d|\xi|\right|\\
    &= 2^{j(2-r)}  \sup_{z\in\R^2} \left| \int_{1/2}^2 \int_{-\epsilon}^\epsilon  e^{\i2^{j}t(\Re\nu(2^j\xi)2^{-j}+z\cdot\xi )} e^{-t\Im \nu(2^j\xi) }\psi(2^j\xi)\chi(|\xi|)|\xi|^{1-r}\d\phi\d|\xi|\right|.
  \end{align*}
  Due to Proposition~\ref{prop3.3} (let us restrict to the $-$ case) we have
  
  \begin{equation}
    \Re\nu(2^j\xi)=2^j\omega(\xi)+2^j|\xi|\alpha(2^j\xi),
  \end{equation}
  where
  \begin{equation}\label{(1)}
     \left|\partial_\phi^k\alpha(2^j\xi)\right|\le C_k\,|\phi|^{2\ell-k},\qquad k=0,1,\ldots,2\ell-1.
  \end{equation}
  By our assumptions $2\ell-1\ge{\bar\gamma}(\bar\eta)$. We use this to reduce the problem to
  properties of the elastic eigenvalue $\omega_j(\eta)$ and consider
  \begin{multline}
     I_j(t)= 2^{j(2-r)}  \sup_{z\in\R^2} \bigg| \int_{1/2}^2 \int_{-\epsilon}^\epsilon  e^{\i t2^{j}|\xi|(\omega(\eta)
    +z\cdot\eta+\alpha(2^j\xi))}  e^{-t\Im \nu(2^j\xi) }\psi(2^j\xi)\d\phi\chi(|\xi|)|\xi|^{1-r}\d\xi\bigg|.
  \end{multline}
  Note that, $|\alpha(2^j\xi)|\lesssim |\phi|^{2\ell}$ is small, if we choose the neighbourhood of 
  $\bar\eta$ small enough.
 
  If $|\omega(\eta)+z\cdot\eta|\ge\delta$ for some small $\delta$ the outer integral has no stationary
  points and integration by parts implies arbitrary polynomial decay. We restrict to the $z\in\R^2$ 
  with $|\omega(\eta)+z\cdot\eta|\le\delta$ and consider only the inner integral
    \begin{equation}
   I_j(t,|\xi|) = \left| \int_{-\epsilon}^\epsilon  e^{\i t2^{j}|\xi|(\omega(\eta)
    +z\cdot\eta+\alpha(2^j\xi))} e^{-t\Im \nu(2^j\xi) }\psi(2^j\xi)\d\phi\right|.
  \end{equation}
  It can be estimated using the Lemma of van der Corput. For this we need an estimate for derivatives of the phase.
  We start by considering the unperturbed phase $\omega(\eta)+z\cdot\eta$. Differentiation yields
 \begin{multline}	
\partial_\eta^{{\bar\gamma}(\bar\eta)}(\omega(\eta)+z\cdot\eta)= \partial_\eta^{{\bar\gamma}(\bar\eta)-2}(\partial_\eta^2\omega(\eta)+\omega(\eta))
-\partial_\eta^{{\bar\gamma}(\bar\eta)-4}(\partial_\eta^2\omega(\eta)+\omega(\eta))+-\\
\cdots+-\begin{cases}
	\omega(\eta)+z\cdot\eta,\qquad &2|{\bar\gamma}(\bar\eta)\\
	\partial_\eta\omega(\eta)+z\cdot\eta^T,&2\!\!\not|{\bar\gamma}(\bar\eta)
        \end{cases}
  \end{multline}
and by Proposition~\ref{prop3.4} the first term is non-zero for $\bar\eta$ while the others are small in a 
neighbourhood of $\bar\eta$. Choosing $\delta$ and the neighbourhood small enough and using \eqref{(1)} this implies
a lower bound on the ${\bar\gamma}(\bar\eta)$-th derivative of the phase,
\begin{equation}
     |\partial_\eta^{{\bar\gamma}(\bar\eta)}(\omega(\eta)+z\cdot\eta+\alpha(2^j\xi))|\gtrsim 1
\end{equation}
uniformly on this neighbourhood of $\bar\eta$ and independent of $j$. Thus we conclude by \cite[p. 334]{Stein} for all 
$t\ge1$ and uniformly in $z$ with $|\omega(\eta)+z\cdot\eta|\le\delta$
   \begin{align}
   I_j(t,|\xi|)\le C_k (t2^{j}|\xi|)^{-\frac1{{\bar\gamma}(\bar\eta)}}\left( 1 +\int_{-\epsilon}^\epsilon
    \left| \partial_\phi e^{-t\Im \nu(2^j\xi)}\psi(2^j\xi)\right|\d\phi\right).
  \end{align}
 To estimate the remaining integral we use Proposition~\ref{prop3.3}
  \begin{align*}
  \int_{-\epsilon}^\epsilon \left| \partial_\phi e^{-t\Im \nu(2^j\xi)}\psi(2^j\xi)\right|\d\phi
  &\lesssim \int_{-\epsilon}^\epsilon e^{-ct\phi^{2\ell}} |\phi|^{2\ell-1} t\d\phi
  =\frac 2c \int_0^\epsilon  e^{-ct\phi^{2\ell}} \phi^{2\ell-1} ct\d\phi\\
  &=-\frac 2c \left[e^{-ct\phi^{2\ell}}\right]_0^\epsilon\lesssim1.
  \end{align*}
Integrating over $\xi\in[1/2,2]$ and  choosing $r\ge2$ we obtain for $j\in\mathbb N_0$
  \begin{equation}\label{(2)}
      I_j(t)\lesssim (1+t)^{-\frac1{{\bar\gamma}(\bar\eta)}},
  \end{equation}
  where the occurring constant is independent of $j$ and \eqref{eq3.19} implies the desired
  $B^2_{1,2}\to B^0_{\infty,2}$ estimate. Interpolation with the $L^2$--$L^2$ estimate
  gives the corresponding $B^r_{p,2}\to B^0_{q,2}$ estimates with $r\ge 2(1/p-1/q)$
  and finally the embedding relations to Sobolev and $L^p$-spaces the desired estimate
  \begin{equation}
     ||e^{\i t\nu(\D)} f||_q\lesssim (1+t)^{-\frac1{{\bar\gamma}(\bar\eta)}} ||f||_{p,r}
  \end{equation}
  for all $f$ satisfying the required Fourier support conditions.
  
\medskip\noindent
{\it Small frequencies.} Assume now that $\hat f$ is supported in a small conical neighbourhood of $\bar\eta$
with $|\xi|\le2$. In this case we estimate directly by the method of stationary phase. We sketch the main ideas.
It is sufficient to estimate
\begin{equation}
  I(t)=\sup_{x\in\R^2} \left| \mathcal F^{-1}\left[e^{\i t\nu(\xi)}\psi(\xi)\chi(|\xi|)\right]\right|,
\end{equation}
where the function $\psi\in C^\infty(\R_+)$ localises to the small neighbourhood of $\bar\eta$ with $|\xi|\le2$. 
Again we require $|\psi(\xi)|+|\partial_\eta\psi(\xi)|+|\partial_{|\xi|}\psi(\xi)|\lesssim1$.
In correspondence to large frequencies $I(t)$ equals to
\begin{align}
 I(t)&=\sup_{z\in\R^2} \left|\int_0^\delta\int_{-\epsilon}^\epsilon e^{\i t|\xi|(\omega(\eta)+z\cdot\eta+\alpha(\xi))}e^{-t\Im\nu(\xi)}  \psi(\xi)\chi(|\xi|)\d\phi|\xi|\d|\xi| \right|.
\end{align}
We distinguish between $|\omega(\eta)+z\cdot\eta|\ge\delta$, where the outer integral has no stationary points,
and $|\omega(\eta)+z\cdot\eta|\le\delta$. In the first case we can apply {\em one} integration by parts and get
$t^{-1}$ using 
\begin{equation}
|\xi|\,|\partial_{|\xi|} e^{-t\Im\nu(\xi)}|\lesssim |\xi|^2 t e^{-t |\xi|^2\phi^{2\ell}}\lesssim 1.
\end{equation}
In the second case we can reduce the consideration to 
\begin{equation}
I(t,|\xi|)= \left| \int_{-\epsilon}^\epsilon e^{\i t|\xi|(\omega(\eta)+z\cdot\eta+\alpha(\xi))}e^{-t\Im\nu(\xi)}\psi(\xi)\d\phi \right|
\end{equation}
and an application of the Lemma of van der Corput. By Propositions~\ref{prop3.3} and \ref{prop3.4} we have
\begin{equation}
  \partial_\phi^{{\bar\gamma}(\bar\eta)} (\omega(\eta)+z\cdot\eta+\alpha(\xi))\ne 0
\end{equation}
for all $\xi$ in a sufficiently small conic neighbourhood of $\bar\eta$. So we obtain
\begin{equation}
  I(t,|\xi|)\lesssim (t|\xi|)^{-\frac1{{\bar\gamma}(\bar\eta)}} \left(1+\int_{-\epsilon}^\epsilon \left|\partial_\phi
  e^{-t\Im\nu(\xi)}\psi(\xi)\right|\d\phi\right)\lesssim  (t|\xi|)^{-\frac1{{\bar\gamma}(\bar\eta)}} .
\end{equation}
Integrating with respect to $|\xi|$ yields the estimate for $I(t)$
\begin{equation}
  I(t)\le \int_0^2 I(t,|\xi|)|\xi|\d|\xi|\lesssim t^{-\frac1{{\bar\gamma}(\bar\eta)}} \int_0^2 |\xi|^{1-\frac1{{\bar\gamma}(\bar\eta)}}\d|\xi|\lesssim (1+t)^{-\frac1{{\bar\gamma}(\bar\eta)}},\qquad t\ge1
\end{equation}
and by H\"older inequality the $L^1$--$L^\infty$ estimate
\begin{equation}
  ||e^{\i t\nu(\D)}f||_\infty\le I(t) ||f||_1\lesssim  (1+t)^{-\frac1{{\bar\gamma}(\bar\eta)}} ||f||_1
\end{equation}
for all $f$ satisfying the required Fourier support condition.

By interpolation with the obvious $L^2$--$L^2$ estimate we get $L^p$--$L^q$ estimates and combining them with the estimate of the first part proves the theorem.
\end{proof}

If the order of contact exceeds the vanishing order of the coupling funtions, the best we can do is to
use the idea of Section~\ref{sec3.1} and to apply standard multiplier estimates. As already remarked 
after the proof of Theorem~\ref{thm3.1} we obtain as decay rate in this case:
\begin{thm}\label{thm3.6}
Let $\bar\eta$ be hyperbolic with respect to the index $j_0$ and
let $a_{j_0}(\eta)$ vanish in $\bar\eta$ of order $\ell$. Let us assume further that
 ${\bar\gamma}_{j_0}(\bar\eta)\ge2\ell$, where ${\bar\gamma}_{j_0}(\eta)$ is defined in Proposition~\ref{prop3.4}. 
 Then 
\begin{equation}
  || e^{\pm\i t\nu_{j_0}(\D)} f ||_{q} \lesssim (1+t)^{-\frac1{2\ell}(\frac1p-\frac1q)} ||f||_{p,r}
\end{equation}
for dual indices $p\in[1,2]$, $pq=p+q$, regularity $r>2(1/p-1/q)$ and 
for all $f$ with $\supp \hat f$ contained in a sufficiently small conical neighbourhood of $\bar\eta$. 
\end{thm} 

\section{Concluding remarks}
{\bf 1.} In a second part of this note, \cite{WirMedia}, we will give concrete applications of the
general treatment presented so far. From the remarks and examples we made in the previous
sections it follows that we indeed cover the estimates of \cite{Borkenstein} for cubic media and 
\cite{Doll} for rhombic media in the situations of coupling vanishing to first order. 

In \cite{WirMedia} we will discuss the situation of higher order tangencies for rhombic media and 
give a new estimate extending the results of \cite{Doll}. 
Furthermore, we will give concrete examples of media for all achievable decay rates in the case of  a differential elastic operator $A(\D)$.

\noindent
{\bf 2.}
In general decay rates of solutions are determined by vanishing properties of the coupling functions. 
If the coupling functions are nonzero, decay rates are parabolic. If they vanish to sufficiently high 
order, decay rates are hyperbolic and determined by the elastic operator (micro-localised to this direction). In the intermediate case simple multiplier estimates are sufficient.

\noindent
{\bf 3.}\label{conclrem}
In the case of isotropic media one of the coupling functions vanishes identically. In this case one pair
of eigenvalues of the matrix $B(\xi)$ is purely real $\nu_\pm(\xi)=\pm(\lambda+\mu)|\xi|$ and the components related to these
eigenvalues solve a wave equation. Thus they satisfy the usual Strichartz type decay estimates, \cite{Bre75},
\begin{equation}\label{StrDecay}
  \| e^{\pm it(\lambda+\mu)|\D|}P^\pm(\D)V_0 \|_q \lesssim (1+t)^{-\frac12(\frac1p-\frac1q)}\|V_0\|_{p,r}
\end{equation}
for $p\in(1,2]$, $pq=p+q$ and $r=2(1/p-1/q)$. More generally, if we know that one pair of
eigenvalues satisfies $\nu_\pm(\xi)=\pm|\xi|\omega(\eta)$ for all $\eta\in\mathbb S^1$ with a smooth function $\omega:\S^1\to\R_+$,
the decay rates for the corresponding components depend heavily on geometric properties of the Fresnel curve
$S=\{\,\omega^{-1}(\eta)\eta\, |\, \eta\in\S^1\,\}\subset\R^2 $. Following \cite{Racke} the estimate \eqref{StrDecay} is valid in 
this case as long as the curvature of $S$ never vanishes. If there exist directions $\eta$ where the curvature of $S$
vanishes, the constant $\frac12$ in the exponent has to be altered to $\frac1{\bar\gamma}$, where $\bar\gamma$ denotes the maximal order
of contact of the curve $S$ to its tangent, \cite{Sug96}.

\noindent
{\bf 4.} The main focus of this paper was on the treatment of non-degenerate cases, thus assumptions (A1) to (A4) are required. Nevertheless, we showed how to obtain a control on the eigenvalues of the matrix $B(\xi)$ in the exceptional cases where either (A3) or (A4) is violated. In the first case the treatment of large frequencies has to be replaced by the investigation of the expression obtained in Proposition~\ref{prop:2.3a}, while in the latter one a corresponding replacement has to be made for small frequencies. 

\noindent
{\bf 5.} The results in this paper are essentially two-dimensional. For the study of anisotropic thermo-elasticity in 
higher space dimensions there arise two essential problems. The first is that we can not assume (A3). For example in 
three space dimensions and for cubic media the symbol of the elastic operator $A:\S^{2}\to\C^{3\times 3}$ has degenerate directions 
with multiple eigenvalues related to the crystal axes. Nevertheless, the multi-step diagonalisation scheme used in 
Section~\ref{sec2.2} can be adapted to such a case. This will be done in the sequel. The second problem is that the
geometry of the set of hyperbolic directions becomes more complicated. In general we can not expect isolated 
hyperbolic directions, it will be necessary to consider manifolds of hyberbolic directions on $\S^2$.

\end{document}